\input ./style/arxiv-general.cfg
\documentclass[aop,MSNbibl,seceqn,dvips]{arximspdf}
\makeatletter
   \@ifpackageloaded{graphicx}{}{\usepackage{graphicx}}
\makeatother
\usepackage{mathbh}

\doi{10.1214/14-AOP959} 
\volume{43}
\issue{6}
\pubyear{2015}
\firstpage{3239}
\lastpage{3278}
\docsubty{FLA}

\makeatletter
\newcommand{\rrVert}{\Vert}
\newcommand{\llVert}{\Vert}
\renewcommand{\mid}{|}
\newcommand{\binom}[2]{{#1 \choose #2}}
\newcommand{\binomm}[2]{\pmatrix{#1\cr #2}}
\newtheorem{theorem}{Theorem}[section]
\newtheorem{lemma}[theorem]{Lemma}
\newtheorem{claim}[theorem]{Claim}
\newtheorem{proposition}[theorem]{Proposition}
\newtheorem{corollary}[theorem]{Corollary}
\newproclaim{example}[theorem]{Example}
\newproclaim{definition}[theorem]{Definition}
\newproclaim{remark}[theorem]{Remark}
\newcommand{\eqref}[1]{(\ref{#1})}
\def\E{{\mathbb E}}
\def\R{{\mathbb R}}
\def\P{{\mathbb P}}
\newcommand{\cC}{{\mathcal C}}
\newcommand{\cF}{{\mathcal F}}
\newcommand{\cG}{{\mathcal G}}
\newcommand{\cI}{{\mathcal I}}
\newcommand{\cP}{{\mathcal P}}
\newcommand{\cS}{{\mathcal S}}
\newcommand{\cT}{{\mathcal T}}
\newcommand{\cW}{{\mathcal W}}
\newcommand{\Dim}{D }
\newcommand{\tv}{{\mathrm{TV}}}
\def\Var{\operatorname{Var}}
\def\eps{\varepsilon}
\def\Cov{\operatorname{Cov}}
\newcommand{\Po}{\operatorname{Po}}
\newcommand{\Bin}{\operatorname{Bin}}
\newcommand{\dist}{\operatorname{dist}}
\newcommand{\aut}{\operatorname{aut}}
\newcommand{\NSTAB}{\mbox{\textsc{Stab}}}
\newcommand{\NS}{\mbox{\textsc{Sens}}}

\newcommand{\SNSn}{$\mbox{\textsc{StrSens}}_{0}$}
\newcommand{\SNSv}{$\mbox{\textsc{StrSens}}_{1}$}
\newcommand{\smP}{\mathrm{P}}
\newcommand{\One}{\mathbh{1}}
\newcommand{\Inf}{\mathbf{I}}
\makeatother

\begin{document}
\begin{frontmatter}

\title{Strong noise sensitivity and random graphs}
\runtitle{Strong noise sensitivity and random graphs}

\begin{aug}
\author[A]{\fnms{Eyal}~\snm{Lubetzky}\ead[label=e1]{eyal@courant.nyu.edu}}
\and
\author[B]{\fnms{Jeffrey E.}~\snm{Steif}\corref{}\ead[label=e2]{steif@chalmers.se}\thanksref{T1}}
\runauthor{E. Lubetzky and J. E. Steif}
\affiliation{Microsoft Research, Chalmers University of Technology and
G\"{o}teborg University}
\address[A]{Courant Institute\\
New York University\\
251 Mercer Street\\
New York, New York 10012\\
USA\\
\printead{e1}}
\address[B]{Mathematical Sciences\\
Chalmers University of Technology\\
G\"oteborg SE-41296\\
Sweden\\
and\\
Mathematical Sciences \\
G\"{o}teborg University \\
G\"{o}teborg SE-41296\\
Sweden\\
\printead{e2}}
\end{aug}
\thankstext{T1}{Supported by the Swedish Research Council and the Knut and Alice Wallenberg Foundation.}

\received{\smonth{8} \syear{2013}}
\revised{\smonth{5} \syear{2014}}

%
\begin{abstract}
The noise sensitivity of a Boolean function describes its likelihood
to flip under small
perturbations of its input. Introduced in the seminal work of
Benjamini, Kalai and Schramm
[\textit{Inst. Hautes \'Etudes Sci. Publ. Math.} \textbf{90} (1999) 5--43],
it was there shown to be governed by the first level of Fourier
coefficients
in the central case of monotone functions at a constant critical
probability $p_c$.

Here we study noise sensitivity and a natural stronger version of it,
addressing the effect
of noise given a specific witness in the original input.
Our main context is the Erd\H{o}s--R\'enyi random graph, where already
the property of
containing a given graph is sufficiently rich to separate these
notions. In particular,
our analysis implies (strong) noise sensitivity in settings where the
BKS criterion
involving the first Fourier level does not apply, for example, when
$p_c \to0$ polynomially fast in the number of variables.
\end{abstract}

%
\end{frontmatter}

\setcounter{footnote}{1}
\section{Introduction}\label{secIntro}
The concept of noise sensitivity, introduced by Benjamini, Kalai and
Schramm~\cite{BKS}, captures the notion that the value of a Boolean
function of many i.i.d. variables would change under small
perturbations of its input. Roughly put, it corresponds to the case
where a small perturbation of the input variables via i.i.d. noise
suffices to make the new value of the function asymptotically
independent of its original value.

Formally, consider a sequence of functions $f_n\dvtx\Omega_n\to\{
0,1\}$
paired with a sequence of probabilities $p_n$, where each domain
$\Omega
_n = \{0,1\}^{\Lambda_n}$ is a product space of $\operatorname{Bernoulli}(p_n)$
variables, and the sets $\Lambda_n$ are finite and increasing with $n$.
Further assume that the sequence $(p_n)$ is \emph{nondegenerate} in
the sense that $\P(f_n=1)$ is uniformly bounded away from $0$ and $1$.
Given $\omega\in\Omega_n$ and some $\varepsilon\in(0,1)$, let
$\omega
^\varepsilon$ denote the result of
resampling the $\operatorname{Bernoulli}(p_n)$ variable $\omega_x$ independently with
probability $\varepsilon$ for each $x\in\Lambda_n$.
The sequence $(f_n)$ is said to be \emph{noise sensitive} ($\NS$)
w.r.t. $p_n$ if for any $\varepsilon>0$,
%
%
\begin{equation}
\label{eqnNS} \lim_{n\rightarrow\infty} \P\bigl(f_{n}\bigl(
\omega^\varepsilon\bigr)=1\mid f_{n}(\omega)=1 \bigr) - \P(
f_{n}=1 )=0,
\end{equation}
or equivalently [recall that $(f_n)$ is nondegenerate], $\Cov
(f_{n}(\omega),f_{n}(\omega^{\varepsilon}) )\to0$.
When a function $(f_n)$ is $\NS$ it is natural to further discuss
\emph{quantitative noise sensitivity}; that is, how fast can
$\varepsilon\to0$
with $n$ such that~\eqref{eqnNS} still holds?

In the setting where $p_n\equiv1/2$ and the functions $f_n$ are \emph
{monotone} w.r.t. the natural partial order on the hypercube $\Omega_n$
(as is notably the case for critical 2\Dim percolation), a beautiful
argument of Benjamini, Kalai and Schramm~\cite{BKS} gave a criterion
for noise sensitivity in terms of the first level of Fourier
coefficients of $f_n$. Namely, $(f_n)$ is noise sensitive if and only
if $\lim_{n\to\infty}\sum_{x\in\Lambda_n}
\hat{f}_n(x)^2 = 0$, where $\hat{f}_n(x)$ is the Fourier coefficient
corresponding to the singleton $\{x\}$, and is also one-half the
probability that $x$ is \emph{pivotal}; that is, flipping its value
would flip the value of $f_n$.
For more on noise sensitivity in this case, see~\cite{GS} and the
references therein.
Unfortunately, this criterion becomes invalid when $p_n\to0$
(e.g., formal definitions postponed, the indicator of a random graph
being triangle-free satisfies the above condition, and yet it is \emph
{not} noise sensitive; see~\cite{BKS}, Section~6.4),
and determining noise sensitivity without it can prove to be a
challenging task already for fairly simple monotone functions enjoying
many symmetries.

\subsection{Strong noise sensitivity}
Going back to~\eqref{eqnNS}, this is known (see~Section~\ref
{NSbackgroundsecFourier}) to be equivalent to having the average of
$\vert\P(f_{n}(\omega^\varepsilon)=1\mid\omega) - \P( f_{n}=1
)\vert$ over $\{\omega\dvtx f_n(\omega)=1\}$ tend to $0$ as $n\to
\infty
$. That is, if $(f_n)$ is noise sensitive, then most inputs $\omega\in
\Omega_n$ with $f_n(\omega)=1$ are such that conditioning on $\omega$
will not give any substantial information on the probability that
$f_n(\omega^\varepsilon)=1$.
When dealing with monotone functions, however, it is in many cases more
natural and useful to condition on a \emph{witness} for $f_n(\omega)=1$
(e.g., a particular crossing in 2\Dim percolation) instead of
the entire configuration $\omega$.

\begin{definition} \label{witness}
A 1-\emph{witness} for a monotone function $f\dvtx\{0,1\}^\Lambda\to
\{0,1\}
$ is a minimal subset $W\subset\Lambda$ such that $\omega_W\equiv1$
implies $f(\omega)=1$.
\end{definition}

Let $\cW_1=\cW_1(f)$ denote the set of 1-witnesses of a monotone
Boolean function~$f$, and let $\cW_0=\cW_0(f)$ denote its analogously
defined 0-witnesses.

Perhaps surprisingly, it can be the case that $(f_n)$ is noise
sensitive and yet the probability that $f_n(\omega^\varepsilon)=1$
substantially increases when we condition on \emph{any particular}
1-witness in $\omega$. This motivates the following definition.

%
\begin{definition} \label{1-strongNS}
A sequence $(f_{n})$ of monotone increasing Boolean functions
is said to be 1-\emph{strongly noise sensitive} (\SNSv) if for any
$\varepsilon>0$,
%
%
\begin{equation}
\label{eqnSNS} \lim_{n\to\infty} \max_{W\in\cW_1}\P
\bigl(f_n\bigl(\omega^{\varepsilon}\bigr)=1\mid
\omega_W\equiv1\bigr) -\P(f_n=1) = 0.
\end{equation}
The notion of 0-\emph{strong noise sensitivity} (\SNSn) is defined
analogously.
[Note that a sequence of increasing functions $(f_n)$ is \SNSn\ if and
only if its complement $(\overline{f_n})$ is \SNSv, where $\overline
{f_n}(\omega) = \overline{f_n(
\bar{\omega})}$ with $\bar{x}=1-x$.]
\end{definition}

As we will later see (and as suggested by its name), the notion of
strong noise sensitivity, which addresses the subtler effect
of conditioning on any \emph{particular} witness [cf.~\eqref{eqnNS}
vs.~\eqref{eqnSNS}],
indeed implies (even when $\varepsilon\to0$) the standard noise
sensitivity but not vice versa.

We now demonstrate this concept through two examples of monotone noise
sensitive functions discussed by Benjamini, Kalai and Schramm in~\cite{BKS},
both of which trace back to Ben-Or and Linial in the related work~\cite{BL}.
\begin{longlist}[(ii)]
\item[(i) \emph{Tribes}.] Partition $\Lambda_n=\{x_1,\ldots,x_n\}$ into
blocks of
$\log_2 n- \log_2 \log_2 n$ variables, let $p_n\equiv1/2$ and set $f_n$
to be $1$ if there is an all-1 block.

It is known~\cite{BKS}, Section~6.1, that this function is
nondegenerate and $\NS$.
\mbox{A~1-}witness $W$ in $\omega$ is a full block, which the noise will
destroy with probability approaching 1,
and the probability of encountering another in $\omega^\varepsilon$ should
be asymptotically $\P(f_n=1)$.
Indeed, tribes is \SNSv.

\item[(ii) \emph{Recursive} 3\emph{-majority}.] Index $n=3^k$
variables by the
leaves of a ternary tree, and iteratively set the value of each node to
be the majority of its children. Take $p_n\equiv1/2$, and define $f_n$
to be the value at the root.

Clearly nondegenerate, this function is known~\cite{BKS}, Section~6.2,
to be $\NS$,
that is, $\P(f_n(\omega^\varepsilon)=1\mid f_n(\omega)=1 )\to1/2$
as $n\to
\infty$.
A 1-witness $W$ is a set of $2^k$ leaves (positioned in the obvious way
to force the majority). It is then easy to verify that
$\P(f_n(\omega^\varepsilon)=1 \mid\omega_W\equiv1 ) =
1-\varepsilon/2$, and
therefore
this function is \emph{not} \SNSv\ (nor \SNSn\ by symmetry).
\end{longlist}
It is important to emphasize the potentially different behaviors of
0-witnesses and 1-witnesses w.r.t. strong noise sensitivity, versus
standard noise sensitivity which is closed under taking complements.
Indeed, by a general principle, the tribes function, mentioned above as
being \SNSv, is \emph{not} \SNSn\ [conditioning on a particular
0-witness in $\omega$ does affect $f_n(\omega^\varepsilon)$ in the limit].

The above examples all featured $p_n\equiv1/2$. Indeed, as noted
in~\cite{BKS}, Section~6.4,

\begin{quotation}
``\emph{When $p$ tends to zero with $n$}, \emph{new phenomena occur.
Consider},
\emph{for example}, \emph{random graphs on $n$
vertices with edge probability} ${p = n^{- a}}\ldots$''
\end{quotation}
Many key features of the Erd\H{o}s--R\'enyi random graph are
nondegenerate at such $p$, and yet the BKS criterion for $\NS$ is then
no longer applicable.


\subsection{Properties of random graphs}
The Erd\H{o}s--R\'{e}nyi random graph, $\cG(n,\break p)$, is a probability
distribution over graphs on $n$ labeled vertices, where each undirected
edge appears independently with probability $p=p(n)$.
A monotone increasing graph property is a collection of graphs closed
under isomorphism and the addition of edges, and we will often identify
it with its indicator function [a~monotone Boolean function on the
$\binom{n}2$ edge variables].

As a first example, consider $\cG(n,p)$ at its famous critical window
centered at $p=1/n$, where the longest cycle is typically of order
$n^{1/3}$; see, for example,~\cite{JLR}.

%
\begin{theorem}\label{thmcube-root-cycle}
Fix $0<a<b$, and let $f_n$ be the property that the critical random
graph $\cG(n,1/n)$ contains a cycle of length $\ell\in(a n^{1/3}, b
n^{1/3})$. Then $(f_n)$ is nondegenerate and noise sensitive, and
furthermore, it is \SNSv.

Moreover, the analogue of this conclusion for quantitative noise
sensitivity holds if and only if the noise parameter $\varepsilon
=\varepsilon
(n)$ satisfies $\varepsilon\gg n^{-1/3}$.
\end{theorem}

Theorem~\ref{thmcube-root-cycle} in fact holds throughout the critical
window $p=\frac{1\pm\xi}n $ with $\xi=O(n^{-1/3})$, around which the
longest cycle grows from constant to linear (e.g., taking $\xi^3 n\to
\infty$ still with $\xi=o(1)$, the\vspace*{1pt} maximum length of a
cycle is $\Theta
_\smP(1/\xi)$ at $p=\frac{1-\xi}n$ and $\Theta_\smP(\xi^2 n)$ at
$p=\frac{1+\xi}n$; see~\cite{JLR}, Theorems 5.17, 5.18).

Revisiting the quantitative conclusion of Theorem~\ref
{thmcube-root-cycle} now highlights an interesting phenomenon, where
the $\varepsilon\gg n^{-1/3}$ threshold for noise sensitivity coincides
with the boundary of the critical window ($p=\frac{1\pm\xi}n$ for
$\xi
\gg n^{-1/3}$). This phenomenon is best explained through the following
equivalent process:
\begin{itemize}
\item Let $\omega$ be a uniform set of $N\sim\Bin(\binom{n}2,p)$
edges.\vspace*{1pt}
\item Obtain $\bar{\omega}$ by deleting a uniform set of $ \Bin
(N,\varepsilon(1-p))$ edges from $\omega$.
\item Add a uniform set of $\Bin(\binom{n}2-N,\varepsilon p)$ edges
missing from $\omega$ to get $\omega^\varepsilon$.
\end{itemize}
As the edge probability in $\bar{\omega}$ is $p(1-\varepsilon
)+\varepsilon
p^2$, on a heuristic level we have:
\begin{longlist}[(a)]
\item[(a)] If $\varepsilon\lesssim n^{-1/3}$, then $\bar{\omega}$
remains in
the critical window, where $(f_n)$ is nondegenerate, so $f_n(\omega
),f_n(\bar{\omega})$ [thus $f_n(\omega),f_n(\omega^\varepsilon)$] should
be correlated.
\item[(b)] If $\varepsilon\gg n^{-1/3}$, then $\bar{\omega}$ is subcritical
whence $f_n(\bar{\omega})$ is degenerate, effectively decorrelating
$f_n(\bar{\omega})$ from $f_n(\omega)$ [thus also $f_n(\omega
),f_n(\omega^\varepsilon)$] yielding $\NS$.
\end{longlist}
Although plausible, it is unclear that in general the degeneracy of
$f_n(\bar{\omega})$ will indeed result in the decorrelation of
$f_n(\omega)$ and $f_n(\omega^\varepsilon)$.

Intuitively, we expect a random graph property to be noise sensitive
when it has no bounded-size witnesses (thus none will survive the noise
in fact), and distinct witnesses are essentially independent (so
surviving fragments of a witness will have negligible impact),
as is the case in the theorem above.

However, for various important graph properties the witnesses happen to
be highly correlated, foiling this intuition. For instance, containing
a Hamilton cycle is nondegenerate at $p\sim\frac{\log n}n$, yet the
expected number of witnesses becomes exponentially large in $n$ already
at $p=O( 1/n)$, and similarly for perfect matchings.
Nevertheless, both are in fact noise sensitive:

%
\begin{theorem}\label{thmmin-degree}
Let $f_n$ be the property that the minimum degree of $\cG(n,p)$ is at
least $k$ for some fixed $k\geq1$, and suppose $p=p(n)$ is such that
$(f_n)$ is nondegenerate.
Then $(f_n)$ is noise sensitive, and moreover, it is \SNSn.

As a result, the following properties of $\cG(n,p)$ are noise sensitive:
\begin{longlist}[(iii)]
\item containing a Hamilton cycle,
\item containing a perfect matching (in general, an $r$-factor\footnote
{An $r$-factor of a graph is a spanning $r$-regular subgraph.} for $r$ fixed),
\item connectivity (in general, $k$-vertex and $k$-edge connectivity
for $k$ fixed),
\item having an isoperimetric constant\footnote{The isoperimetric
constant of a graph is the minimum of $\frac{e(S,S^c)}{\vert
S\vert\wedge\vert S^c\vert}$ over all subsets $S$ of the
vertices, where $e(S,S^c)$ is the number of edges between $S$ and its
complement.} of at least $\gamma$ for some fixed $\gamma>0$.
\end{longlist}
Furthermore, each of these is quantitatively noise sensitive if and
only if $\varepsilon\gg\frac{1}{\log n}$.
\end{theorem}

It is worthwhile noting that not even the (nonstrong) noise
sensitivity in Theorems~\ref{thmcube-root-cycle} or~\ref{thmmin-degree} can be obtained from the best-known generalizations of
the BKS criterion for varying $p$ (see~\cite{KK}), as these all require
$1/p=n^{o(1)}$.

We turn our attention to the well-studied family of properties of the
form ``$\cG(n,p)$ contains a copy of a given graph $H_n$.'' Obviously,
if the size of $H_n$ is uniformly bounded then this property is \emph
{not} noise sensitive, since a copy of $H_n$ will survive the noise
with positive probability (as noted in~\cite{BKS}, Section~6.4, it is
\emph{noise stable},
a notion basically the opposite of being noise sensitive).
Note that having the number of edges in $H_n$ grow with $n$ is a
necessary but not sufficient condition for noise sensitivity (e.g.,
take $\log n$ disjoint edges).

The case where $H_n$ is a clique concerns the maximum clique size in
$\cG(n,p)$. It is well known (see, e.g.,~\cite{AS}) that at $p=1/2$
this concentrates on a single point $k_n\sim2 \log_2 n$ for most values
of $n$, while for exceptional values of $n$ it is either $k_n$ or
$k_n+1$ with high probability. In the latter case, one can ask whether
the property that $k_n$ is the maximum clique size is noise sensitive.
Indeed it is, as implied by the BKS criterion; see~Section~\ref
{secprelimcliques}. However, one would expect there to be a direct
proof of this fact that does not employ the machinery of Fourier
analysis and hyper-contractive estimates.

Here we provide a direct proof of strong noise sensitivity for this
property. 

%
\begin{theorem}
\label{thmclique}
Let $f_n$ be the property that $\cG(n,p)$ has a clique of size $k_n$
for $k_n=n^{o(1)}$ such that $k_n\to\infty$ with $n$, and suppose
$p=p(n)$ is such that $(f_n)$ is nondegenerate. Then $(f_n)$ is noise
sensitive. Moreover, it is \SNSv.
\end{theorem}

Consider the above theorem for $1 \ll k_n \lesssim\log n$. When $H_n$
is a clique of size $k_n$, containing $H_n$ in $\cG(n,p)$ is $\NS$.
However, if $H_n$ consists of $k_n$ disjoint edges for the same
sequence $k_n$, then the property is noise stable (essentially as a
majority function).
In light of these two opposite behaviors, one wishes to understand
which features of the given graph $H_n$ dictate $\NS$.

While determining noise sensitivity for graphs $H_n$ whose size grows
rapidly with $n$ can be delicate, the picture is fairly well understood
when the graph sizes are at most a certain poly-log of $n$.
In that case, it turns out that a single feature of $H_n$---being \emph
{strictly balanced}---governs noise sensitivity.
A graph is \emph{balanced} if its average degree is at least that of
any of its proper subgraphs, and
it is \emph{strictly balanced} if these inequalities are all strict
(e.g., a clique is strictly balanced whereas a collection of disjoint
edges is balanced).


%
\begin{theorem}\label{thm-strictly-balanced}
Let $H_n$ be a sequence of graphs, and let $f_n$ be the property that
the random graph $\cG(n,p)$ contains a copy of $H_n$. The following holds:
\begin{longlist}
\item[(1)] If $H_n$ is strictly balanced
with $1 \ll\ell_n \leq(\frac{\log n}{\log\log n} )^{1/2}$ edges,
then $(f_n)$ is noise sensitive, and furthermore, it is \SNSv.
\item[(2)] There exists a sequence of
strictly balanced graphs $H_n$ with $\ell_n \asymp\log n$ edges for
which $(f_n)$ is not noise sensitive.
\end{longlist}
\end{theorem}

We stress that the assumption that $H_n$ is strictly balanced is
necessary in the sense that without it, one could take $H_n$ to be
$\ell
_n$ disjoint copies of any fixed strictly balanced graph (e.g., a
clique or a tree) for any $ \ell_n \ll\sqrt{n}$, whence containing
$H_n$ is not $\NS$ (in fact, it is noise stable). However, not that
having $H_n$ be strictly balanced is a necessary condition for $\NS$;
for example, we will see that containing a disjoint union of two
cliques is \SNSv.

The last two theorems will be obtained as a consequence of a general
tool (Proposition~\ref{proppoisson-WD}) which deduces \SNSv\ from an
appropriate Poisson approximation of the number of copies of $H_n$ in $G$.

We note that each of the properties shown in Theorems~\ref
{thmcube-root-cycle}--\ref{thm-strictly-balanced} to be \SNSv\ is
\emph{not} \SNSn,
and the properties that were shown to be \SNSn\ are \emph{not} \SNSv.
Indeed, a general principle (Lemma~\ref{lemboundedwitness}) will yield
that if we let $X_n$ denote the number of 1-witnesses $W$ for which
$\omega_W\equiv1$, then having $\E[X_n]=O(1)$ precludes \SNSn\ (and
similarly for 0-witnesses). At the same time, there can be monotone
Boolean functions that are both \SNSn\ and \SNSv, as we demonstrate
in Section~\ref{sec0-vs-1}.

\subsection{Organization}
The rest of the paper is outlined as follows.
In Section~\ref{secprelim}, we provide prerequisites on noise
sensitivity. Section~\ref{secwit-to-noise} demonstrates the use of
strong noise sensitivity toward establishing noise sensitivity,
including the proof of Theorems~\ref{thmcube-root-cycle} and~\ref
{thmmin-degree}. Section~\ref{secwit-dist} looks into the
dependencies between witnesses for a sufficient condition for strong
noise sensitivity. This condition is then applied in the context of
containing a given graph in $\cG(n,p)$ and in particular toward the
proofs of Theorems~\ref{thmclique} and~\ref{thm-strictly-balanced}.
Finally, Section~\ref{sec0-vs-1} compares the 0-strong and \mbox{1-}strong
noise sensitivity of a function, as well as the validity of these
properties under varying levels of noise. 

\section{Preliminaries}\label{secprelim}

This section includes background on noise sensitivity, both for constant
$p$ and when the probabilities $p$ are allowed to vary with $n$; see,
for example,~\cite{KK} for additional information on this topic.
We first set some standard notation.

\subsection{Notation}
Throughout the paper, a sequence of events $A_n$ is said to hold with
high probability (w.h.p.) if $\P(A_n)\to1$ as $n\to\infty$. We use the
notation $f=O_\smP(g)$ to denote that the ratio $f/g$ is bounded in
probability, and the analogous $f=\Theta_\smP(g)$ to denote that
$f=O_\smP(g)$ and $g=O_\smP(f)$.
At times we use $f \ll g$ and $f \lesssim g$ to abbreviate $f = o(g)$
and $f=O(g)$, respectively, as well as the converse form of these.
We will often omit the subscript $n$ from the probabilities $p_n$ under
consideration in this paper (though these will typically tend to 0 as
$n\to\infty$) for simplicity.

\subsection{Influences and the pivotal set}\label{NSbackgroundsecFourier}
The notion of influence, defined next, is fundamental in the study of
noise sensitivity of functions.

\begin{definition}
Given a Boolean function $f$ from $\Omega= \{0,1\}^{\Lambda}$ into $\{
0,1\}$, $p\in(0,1)$
and $i\in\Lambda$, the \emph{influence of $i$ with respect to $p$} is
defined to be
%
%
\begin{equation}
\label{eq-def-influence} \Inf_i(f) = \P\bigl(f(\omega)\neq f\bigl(
\omega^i\bigr)\bigr),
\end{equation}
where $\omega^i$ is $\omega$ flipped in the $i$th coordinate.
\end{definition}

(As usual, the above definition implicitly depends on $p$ through $\P$.)
The following theorem of~\cite{BKS} is one of the central results on
noise sensitivity.

%
\begin{theorem}[(\cite{BKS})]\label{thNSmainresult}
Let $p_n\equiv p$ for some fixed $0<p<1$.
If
%
%
\begin{equation}
\label{eSquaredInfluencesTo0} \lim_{n\to\infty}\sum_i
\Inf_i(f_n)^2 = 0
\end{equation}
for a sequence of Boolean functions $(f_n)$, then $(f_n)$ is $\NS$.
\end{theorem}

As we will see below, for monotone functions and constant $p$
the converse is also true, while what occurs when $p_n\to0$ is more subtle.

Consider the random set of \emph{pivotal} variables defined as
\[
\cP(\omega):=\cP_f(\omega):= \bigl\{i\in\Lambda\dvtx f(\omega
)\neq f
\bigl(\omega^i\bigr) \bigr\}.
\]
[Notice $\P(i\in\cP)=\Inf_i$.]
The following easy lemma will be used in this paper.

%
\begin{lemma} \label{lemmaPivotalFormula}
Every monotone Boolean function $f$ satisfies
\[
\E\bigl[\vert\cP\vert\mid f=1\bigr]=\frac{p}{\P(f=1)}\E\vert
\cP
\vert.
\]
\end{lemma}

\begin{pf}
Note that $\{f(\omega)\neq f(\omega^i)\}$ and $\{\omega_i=1\}$ are
independent, so the left-hand side of the desired equality is easily
seen to be equal to
\[
\sum_i \P\bigl(f(\omega)\neq f\bigl(
\omega^i\bigr)\mid f=1\bigr)= 
\sum
_i \frac{\P(f(\omega)\neq f(\omega^i), \omega_i=1)}{\P(f=1)}=
\frac
{p}{\P(f=1)}\E\vert\cP\vert,
\]
where the first equality uses monotonicity, and the second
equality uses the earlier stated independence.
\end{pf}

\begin{remark*}
The above also holds for nonmonotone functions when $p=1/2$.
\end{remark*}

We now indicate that the equivalence holding for monotone functions and
constant $p$ between
$\sum_i \Inf_i(f_n)^2 = o(1)$, and $\NS$ in fact fails for varying $p$
in either direction.
Let $f_n$ be the indicator function of a random graph containing a copy
of~$K_4$ with $p=n^{-2/3}$.
Clearly $\E[\vert\cP\vert\mid f=1]\le6$ which\vspace*{1pt} by
Lemma~\ref{lemmaPivotalFormula} implies that
$\E\vert\cP\vert= O(n^{2/3})$. By symmetry, this yields $I_i =
O(n^{-4/3})$ for each $i$,
which easily yields~\eqref{eSquaredInfluencesTo0}, and yet this
sequence is clearly stable.
On the other hand, if $f_n$ is the indicator function of a random graph
with $p=\frac{\log n}{n}$ having
minimal degree 1, then $\{f_n\}$ is $\NS$; see Theorem~\ref
{thmmin-degree}. However,
it is easy to verify that $\E[\vert\cP\vert\mid f=0]\gtrsim n$, which
by Lemma~\ref{lemmaPivotalFormula} yields $\E\vert\cP\vert
\gtrsim n$ and so $\sum_i \Inf_i(f_n)^2 \gtrsim1$.

We will see in the next subsection that asking about a possible
equivalence of
$\sum_i \Inf_i(f_n)^2 = o(1)$ and $\NS$ is in fact not really the right
question: instead one should ask about
a possible equivalence of $p\sum_i \Inf_i(f_n)^2 = o(1)$ and $\NS$.

\subsection{Fourier analysis}\label{Fourierbackground}
Fourier analysis is usually a crucial tool
in studying noise sensitivity. We give a quick presentation of this.
From it, one readily sees
some of the basic properties of noise sensitivity.

For a set $\Lambda$, $\omega\in\{0,1\}^{\Lambda}$ and $i\in\Lambda$,
we define
\[
\chi_i(\omega)=\cases{ \displaystyle\sqrt{(1-p)/p}, &\quad if $
\omega_i=1$,
\cr
\displaystyle-\sqrt{p/(1-p)}, &\quad if $
\omega_i=0$.}
\]
Furthermore, for $S\subseteq\Lambda$, let $\chi_S(\omega):=\prod
_{i\in S}\chi_i(\omega)$.
(In particular, $\chi_\varnothing$ is the constant function 1.)
The set $\{\chi_S\}_{S\subseteq\Lambda}$ forms an orthonormal basis
for the set
of functions $f\dvtx\{0,1\}^{\Lambda}\mapsto\R$ when the latter is
equipped with the inner product
$ \langle f,g\rangle:=\E[fg] $
(recall there is always an implicit $p$ when we write $\P$ or
$\E$). We can therefore expand such functions
$
f(\omega)=\sum_{S\subseteq\Lambda} \hat{f}(S)\chi_S(\omega)
$,
where $\hat{f}(S):=\E[f\chi_S]$ is the Fourier--Walsh coefficient of $f$.
Note that $\hat f(\varnothing)$ is the average $\E f$ and by
Parseval's formula
$
\E[f^2] = \sum_{S\subseteq\Lambda} \hat f(S)^2
$.
This orthogonal basis turns out to be an extremely useful one for
studying noise sensitivity,
as the following easily verified formula demonstrates:
%
%
\begin{equation}
\label{ecorrelationFourier} \E\bigl[f(\omega) f\bigl(\omega^{\eps
}\bigr
)\bigr] = \sum
_S \hat f(S)^2 (1-
\eps)^{\vert S\vert}.
\end{equation}
This yields
\[
\Cov\bigl(f_{n}(\omega),f_{n}\bigl(
\omega^{\varepsilon}\bigr) \bigr) = \sum_{S\neq\varnothing} \hat
f(S)^2 (1-\eps)^{\vert S\vert}.
\]

The following theorem now follows immediately; note importantly how it
shows that if the appropriate covariance
goes to 0 for one value of $\varepsilon$, then it does so for all
$\varepsilon
$. Note that there is no condition
on the sequence $(p_n)$.

%
\begin{theorem} \label{thmBKS} Let $(f_n)$ be a sequence of Boolean
functions. Then $(f_n)$ is $\NS$ if and only if any one of the
following conditions holds:
\begin{longlist}[(3)]
\item[(1)] For some $0<\varepsilon<1 $ we have
$\lim_{n\to\infty} \sum_{S\neq\varnothing} \hat{f}_n(S)^2
(1-\eps
)^{\vert S\vert}=0$.
\item[(2)] For every $0<\varepsilon<1 $ we have
$\lim_{n\to\infty} \sum_{S\neq\varnothing} \hat{f}_n(S)^2
(1-\eps
)^{\vert S\vert}=0$.
\item[(3)] For every $k$ we have $\lim_{n\to\infty} \sum
_{0<\vert S\vert<k}\hat{f}_n(S)^2=0$.
\end{longlist}
\end{theorem}

A very useful mnemonic device is the so-called spectral sample $\cS
=\cS
_f$ of a~Boolean
function $f$, defined distributionally by
\[
\P(\cS= S):=\hat f(S)^2\qquad(S \subset\Lambda).
\]
The total weight of this distribution is less than 1 (unless $f\equiv1$).
Note that the terms in items~(1) and~(3) in
Theorem~\ref{thmBKS}, respectively, become
\[
\E\bigl[(1-\varepsilon)^{\vert\cS_n\vert} \One_{\{\cS\neq
\varnothing
\}} \bigr]\quad\mbox{and}\quad\P
\bigl(0< \vert\cS_n\vert< k\bigr).
\]

It turns out that $\NS$ is equivalent to another condition---appearing
perhaps stronger at first glance---according to which for most $\omega$
with $f_n(\omega)=1$, the conditional probability that
$f_n(\omega^\varepsilon)=1$ given $\omega$ is close to the unconditional
probability.

%
\begin{proposition} \label{proequivalentNSdefn}
Let $(f_n)$ be a sequence of Boolean functions. Then $(f_n)$ is $\NS$
if and only if any one of the following conditions holds:
\begin{longlist}[(2)]
\item[(1)] $ [\P(f_n(\omega^\varepsilon)=1 \mid
\omega) -\P(f_n(\omega)=1) ]\stackrel{\mathrm{p}}{\to} 0$.
\item[(2)] $ [\P(f_n(\omega^\varepsilon)=1
\mid\omega) -\P(f_n(\omega)=1) ]
\One_{\{f_n(\omega)=1\}}\stackrel{\mathrm{p}}{\to}0$.
\end{longlist}
\end{proposition}

\begin{pf}
It is immediate that~(1) implies~(2).
To see that~(2) implies $\NS$ as
per~\eqref{eqnNS}, simply write the expression appearing in~\eqref{eqnNS}
as
\[
\sum_{\omega\dvtx f_n(\omega)=1} \bigl[\P\bigl(f_{n}\bigl(
\omega^\varepsilon\bigr)=1\mid\omega\bigr) - \P( f_{n}=1 )
\bigr] \frac{\P(\omega)}{\P( f_{n}=1 )}.
\]
It remains to show that $\NS$ implies~(1).
It is easy to verify that
\[
\Var\bigl(\P\bigl(f_{n}\bigl(\omega^\varepsilon\bigr)=1\mid
\omega\bigr) \bigr) =\sum_{S\neq\varnothing}
\hat{f}_n(S)^2 (1-\eps)^{2\vert S\vert}.
\]
Therefore, by Theorem~\ref{thmBKS}, if $(f_n)$ is $\NS$, we can infer that
\[
\lim_{n\to\infty}\Var\bigl(\P\bigl(f_{n}\bigl(
\omega^\varepsilon\bigr)=1\mid\omega\bigr)\bigr)=0.
\]
Since $\E[\P(f_{n}(\omega^\varepsilon)=1\mid\omega)]=\P(f(\omega)=1)$,
this immediately gives~(1).
\end{pf}

While Theorem~\ref{thmBKS} is quite easy, Theorem~\ref{thNSmainresult}
is much deeper.
It turns out that the converse of Theorem~\ref{thNSmainresult} with
constant $p$ is true for
monotone functions as we now explain. First, for a monotone Boolean
function $f$ mapping into
$\{0,1\}$, one can easily check that
%
%
\begin{equation}
\label{eqnFourierInfluenceRelationship} \hat f\bigl(\{i\}\bigr
)=\sqrt
{p(1-p)} \Inf_i(f).
\end{equation}
This formula together with Theorem~\ref{thmBKS} immediately yields the converse
of Theorem~\ref{thNSmainresult} for fixed $p$. This reinterprets
Theorem~\ref{thNSmainresult} in the monotone case as saying that for
constant $p$,
if the ``sum of the squares of the level 1 Fourier coefficients''
$\sum_{\vert S\vert=1}\hat{f}_n(S)^2$ approaches 0, then the
sequence in $\NS$.

We now consider Theorem~\ref{thNSmainresult} in the context of varying
$p$, in particular for $p$ tending
to 0 with $n$. As above, for monotone functions,~\eqref
{eqnFourierInfluenceRelationship} and
Theorem~\ref{thmBKS} yield the fact that for arbitrary $(p_n)$, $\NS$ implies
%
%
\begin{equation}
\label{eSquaredInfluencesTo0WITHp} \lim_{n\to\infty}p(1-p)\sum
_i \Inf_i(f_n)^2 = 0.
\end{equation}
From this discussion, it follows that the version of Theorem~\ref
{thNSmainresult} that one might hope for,
for arbitrary $(p_n)$, is that~\eqref{eSquaredInfluencesTo0WITHp}
implies $\NS$; equivalently, for monotone
functions, convergence of the level 1 Fourier coefficients implies $\NS$.
Unfortunately, this is not true as we saw in the previous subsection
for the event ``containing a $K_4$.''
Alternatively, if we let $p_n=1/n$ and consider the indicator function of
containing a triangle, then it is easy to see that this sequence is not
$\NS$ (and in fact noise stable,
see this definition below) although~\eqref{eSquaredInfluencesTo0WITHp}
is of order $1/n$.
The stability of the indicator function $f_n$ for containing a triangle
implies that
$
\lim_{k\to\infty} \sup_{n} \sum_{\vert S\vert\ge k}\hat{f}_n(S)^2=0
$.
In addition, in~\cite{FR} it is shown that for any $k\not\equiv
0\pmod
{3}$, this $f_n$ satisfies
\[
\lim_{n\to\infty} \sum_{\vert S\vert= k}
\hat{f}_n(S)^2=0;
\]
that is, the Fourier weights are concentrated on levels $0,3,6,\ldots$
but stay near 0.
(Such a thing cannot occur for monotone functions with constant $p$.)

We end this subsection by
defining the closely related (but opposite) concept to $\NS$, namely
\emph{noise stability}.

\begin{definition} \label{defstab1}
The sequence of functions $f_{n}\dvtx\{0,1\}^{\Lambda_n}\rightarrow\{
0,1\}$
is noise stable ($\NSTAB$) if for any $\delta>0$, there exists an
$\varepsilon>0$ such that
\[
\sup_n \P\bigl(f_{n}(\omega)\neq
f_{n}\bigl(\omega^{\varepsilon}\bigr)\bigr)\le\delta.
\]
\end{definition}

If $\varepsilon_n \to0$ with $n$, one can talk about $\NSTAB$ with
respect to $\{\varepsilon_n\}$ in the obvious way.
Note that while \SNSv\ and $\NS$ with respect to a sequence $\{
\varepsilon
_n\}$ going to 0 is stronger than
ordinary \SNSv\ and $\NS$, $\NSTAB$ with respect to such a sequence is
weaker than ordinary $\NSTAB$.

\subsection{Relation to coarse and sharp thresholds}

It is natural to wonder where the important results in \cite{FR}
concerning sharp
thresholds fall into the context of this paper.
In short, they occur in a very different regime. To explain this,
consider for the moment $p=1/2$.
There are three common scenarios that can occur (as well as various
combinations):

\begin{longlist}[(1)]
\item[(1)] $\E\vert\cS_n\vert= O(1)$.
\item[(2)] $\E\vert\cS_n\vert\to\infty$, and yet $\vert\cS
_n\vert$ is bounded in probability.
%
\item[(3)]
For every fixed $k$ we have $\P(0<\vert\cS_n\vert<k)\to0$,
that is, $(f_n)$ is $\NS$.
\end{longlist}

The first scenario occurs, for example, if $f_n$ only depends on a
fixed finite number of variables independent
of $n$. An example where the second scenario occurs is the sequence of
majority functions. Similar
to~\eqref{eSquaredInfluencesTo0WITHp}, there is another relationship
between influences and the Fourier
picture which does not require monotonicity.
%
This states that $ \sum_{S}\hat{f}(S)^2\vert S\vert=p(1-p)\sum_i
\Inf_i(f)$,
or equivalently,
%
%
\begin{equation}
\label{eFourAndInfluencesGeneralp} \E\vert\cS\vert=p(1-p)\E
\vert\cP\vert
\end{equation}
(as was established for $p=1/2$ in~\cite{KKL}; the case of general $p$
follows similarly).

In~\cite{FR}, results of the form that if you are in the first
scenario, then for graph properties,
the function can be well approximated by functions which depend on a
fixed number of graphs.
Since the context of~\cite{FR} was $p=o(1)$, in view of~\eqref
{eFourAndInfluencesGeneralp},
the assumptions in~\cite{FR} are of the form $p\sum_i \Inf_i(f)\le C$.

\subsection{Maximum cliques in random graphs}\label{secprelimcliques}

As mentioned above, the maximum clique of $\cG(n,p)$ for $p=1/2$
concentrates on 1 point for most values of $n$, yet for infinitely many
values of $n$ it is
concentrated on 2 points. It is for the latter values of $n$ that we
have a nondegenerate indicator function
corresponding to the event that we contain a clique of size about
$k_n\sim2 \log_2 n$.
We describe here how Theorem~\ref{thNSmainresult} yields $\NS$, as was
indicated by Jeff Kahn. Consider the expected size of $\cP_n$ (the set
of pivotal edges). Since $p=1/2$, Lemma~\ref{lemmaPivotalFormula} gives
\[
\E\vert\cP_n\vert=2\P(f_n=1)\E\bigl[\vert
\cP_n\vert\mid f_n=1\bigr].
\]
Hence,\vspace*{1pt} for the nondegenerate $n$ we focus on,
$\E\vert\cP_n\vert$ and $\E[\vert\cP_n\vert\mid f_n=1]$
are of the same order. Clearly whenever $f_n=1$
necessarily $\vert\cP_n\vert=O(\log^2 n)$ since if there is at
least one clique, one can choose
such a clique arbitrarily and then observe that any pivotal edge must
belong to it.
This shows that $\E\vert\cP_n\vert=O(\log^2 n)$, and hence the
influence of each edge
is of order at most $(\frac{\log n}{n})^2$. Squaring this and
multiplying by the number of edges, one obtains that
$
\sum_i \Inf_i(f_n)^2 \lesssim(\log n)^4 /n^2
$.
Since this approaches 0 with $n$, Theorem~\ref{thNSmainresult} yields
noise sensitivity.

\section{From witnesses to noise sensitivity}\label{secwit-to-noise}
In this section we relate noise sensitivity to strong noise
sensitivity. Via this connection we prove quantitative versions of
Theorems~\ref{thmcube-root-cycle} and~\ref{thmmin-degree}.

\subsection{Strong noise sensitivity}\label{subsecstrnoise}
We begin with a straightforward lemma showing that strong noise
sensitivity indeed implies the standard one.

\begin{lemma} \label{lemmaSNStoNS}
Let $(f_{n})$ be a nondegenerate sequence of monotone Boolean functions.
If $(f_n)$ is \SNSv, then it is
noise sensitive. Furthermore, \SNSv\ w.r.t. $\varepsilon=\varepsilon
(n)\to0$
implies quantitative $\NS$ w.r.t. the same $\varepsilon$.
\end{lemma}

\begin{pf}
By the definition of noise sensitivity in~\eqref{eqnNS}, we aim to
show that
\[
\P\bigl(f_n\bigl(\omega^{\varepsilon}\bigr)=1\mid
f_n(\omega)=1 \bigr) - \P(f_n=1) \to0
\]
as $n\to\infty$, where $\varepsilon=\varepsilon(n)$ is allowed to
tend to $0$
with $n$.
By the FKG inequality we have $\P(f_n(\omega^{\varepsilon})=1\mid
f_n(\omega)=1 ) \geq\P(f_n=1)$, and it remains to provide the
corresponding upper bound. Let $\cW_1=\{W_1,\ldots,W_{m_n}\}$ be the
1-witnesses for $f_n$ (arbitrarily ordered), and define the variable
$J$ to be
\[
J = \min\{ 1 \leq j \leq m_n\dvtx \omega_{W_j} \equiv1\}
\]
or $\infty$ in case $f_n(\omega)=0$. With this notation,
%
%
\begin{eqnarray}\label{eqNeededLimitAgain}
&& \P\bigl(f_n\bigl(\omega^{\varepsilon}\bigr)=1\mid
f_n(\omega)=1\bigr)
\nonumber\\[-8pt]\\[-8pt]\nonumber
&&\qquad = \sum_{j=1}^{m_n}
\P\bigl(f_n\bigl(\omega^{\varepsilon}\bigr)=1\mid J=j\bigr) \P
\bigl(J=j\mid f_n(\omega)=1\bigr),
\end{eqnarray}
and again by FKG we see that
\[
\P\bigl(f_n\bigl(\omega^{\varepsilon}\bigr)=1\mid J=j \bigr)\leq
\P\bigl(f_n\bigl(\omega^{\varepsilon}\bigr)=1\mid
\omega_{W_j} \equiv1 \bigr)
\]
since we can condition on $\{J=j\}$ by first conditioning on
$\{\omega_{W_j} \equiv1 \}$ (obtaining a positively associated measure
which enjoys the FKG inequality) and then further conditioning
on the decreasing event $\bigcap_{j'<j} \{ \omega_{W_{j'}} \not
\equiv1
\}$. The\vspace*{1pt} latter can only decrease
the probability of the increasing event $\{f_n(\omega^{\varepsilon
})=1\}$;
thus the last display is established, and altogether we obtain that
%
%
\begin{equation}
\label{eq-ns-leq-sns} \P\bigl(f_n\bigl(\omega^{\varepsilon}\bigr
)=1\mid
f_n(\omega)=1 \bigr) \leq\max_{W\in\cW_1} \P
\bigl(f_n\bigl(\omega^{\varepsilon}\bigr)=1\mid
\omega_{W} \equiv1 \bigr).
\end{equation}
Subtracting $\P(f_n=1)$, and taking $n\to\infty$ now completes the
proof by the definition of \SNSv\ in~\eqref{eqnSNS}.
\end{pf}

%
\begin{remark}\label{remessential-str-NS}
The proof that strong noise sensitivity implies the standard one, in
fact requires a slightly weaker condition than the one stated in~\eqref
{eqnSNS}. Instead of having $\max_W[\P(f_n(\omega^\varepsilon
)=1\mid
\omega_W\equiv1 )-\P(f_n=1)] \to0$, we only need an expectation over
this quantity
w.r.t. a certain distribution over the witnesses (the first $W$ to
appear according to some ordering) to vanish.

In particular, Lemma~\ref{lemmaSNStoNS} remains valid under the
analogue of~\eqref{eqnSNS} for all witnesses $W$ except some subset
$\cW_1^*\subset\cW_1$
with $\P(\bigcup_{W\in\cW_1^*}\{\omega_W\equiv1\})\to0$.
\end{remark}

\begin{example*}[(Tribes)]
Recalling the definition of the tribes function from the \hyperref[secIntro]{Introduction},
a 1-witness $W\in\cW_1$ is a full block. Writing
\begin{eqnarray*}
&& \P\bigl(f_n\bigl(\omega^\varepsilon\bigr)=1\mid
\omega_W\equiv1 \bigr)
\\
&&\qquad \leq\P\biggl(\bigcup
_{W' \neq W} \bigl\{\omega^\varepsilon_{W'}\equiv1
\bigr\} \Big|\omega_W\equiv1 \biggr)
+ \P\bigl(\omega^\varepsilon_{W}\equiv1 \mid
\omega_W\equiv1 \bigr),
\end{eqnarray*}
the last term is equal to $(1-\varepsilon/2)^{\vert W\vert}\to0$
as we have $\vert W\vert\sim\log_2 n \to\infty$ with $n$, while
the first term on the right-hand side is equal to
\[
\P\biggl(\bigcup_{W' \neq W} \{\omega_{W'}
\equiv1\} \biggr) \leq\P(f_n=1)
\]
since any two distinct witnesses $W,W'$ are disjoint, and thus $\{
\omega
_W\equiv1\}$ and $\{\omega_{W'}\equiv1\}$ are independent.
This establishes that
\[
\limsup_{n\to\infty}\max_W \bigl[\P
\bigl(f_n\bigl(\omega^\varepsilon\bigr)=1\mid
\omega_W\equiv1 \bigr)-\P(f_n=1)\bigr] \leq0,
\]
and since it is always nonnegative (by a monotonicity argument), we
conclude that the tribes function is \SNSv.
\end{example*}

\begin{example*}[(Recursive majority)]
Consider\vspace*{1pt} first the canonical 1-witness $W$ for the recursive 3-majority
of $n=3^k$ variables (i.e., $W$ repeatedly reveals the first 2 of the 3
children of a vertex).
Recalling that $p=1/2$, the quantity
%
%
\begin{equation}
\label{eq-zeta-recusive-maj} \zeta_k^\varepsilon= \P\bigl(f_n
\bigl(\omega^\varepsilon\bigr)=1\mid\omega_W\equiv1 \bigr)
\end{equation}
is easily seen (by the nature of this recursive definition) to satisfy
\[
\zeta_{k}^\varepsilon= \bigl(\zeta_{k-1}^\varepsilon
\bigr)^2 + 2\zeta_{k-1}^\varepsilon\bigl(1-
\zeta_{k-1}^\varepsilon\bigr)p = \zeta_{k-1}^\varepsilon,
\]
thus $\zeta_k^\varepsilon= \zeta_0^\varepsilon= 1-\varepsilon/2$
for any $k$.
In particular, recursive 3-majority is not \SNSv\ despite the fact that
it is noise sensitive [indeed, it is easy to see that the influence of
a variable is $2^{-k}$, and so the sum of squared influences is
$(3/4)^k$ which vanishes as $k\to\infty$, satisfying the BKS criterion
for $\NS$].

We emphasize that for this function not only is $\P(f_n(\omega
^\varepsilon
)=1\mid\omega_W\equiv1)$ bounded away from $\P(f_n=1)=1/2$ (enough in
itself to preclude \SNSv), but rather it is $1-\delta(\varepsilon)$ where
$\delta(\varepsilon)\to0$ with $\varepsilon$. This resembles the
notion of
noise stability [where $\P(f_n(\omega^\varepsilon)=1\mid f_n(\omega)=1)$
approaches $1$ as $\varepsilon\to0$].

Interestingly, further increasing the size of the majority yields an
even stronger witness dependency.
As before $\P(f_n(\omega^{\varepsilon})=1\mid\omega_W\equiv1)\geq
1-\delta
(\varepsilon)$, but instead of $\delta(\varepsilon)=\varepsilon/2$
(the case for
3-majority), we now have $\delta(\varepsilon)=o(1)$.

%
\begin{claim} \label{clm5Majority}
Let $f_{n}$ be the recursive 5-majority function on $n=5^k$ vertices.
Then for every
$0<\varepsilon<1$,
\[
\lim_{n\to\infty} \inf_{W\in\cW_1}\P
\bigl(f_n\bigl(\omega^{\varepsilon}\bigr)=1\mid
\omega_W\equiv1\bigr)=1.
\]
\end{claim}

\begin{pf}
As before, consider the canonical 1-witness $W$ which repeatedly
specifies $3$ of $5$ children of a vertex, and define $\zeta
_k^\varepsilon
$ as in~\eqref{eq-zeta-recusive-maj}.
In this way, conditioned on $W$, the root has $3$ children each of
which is a $\operatorname{Bernoulli}(\zeta_{k-1})$ and $2$ other children which are
$\operatorname{Bernoulli}(1/2)$. It is then easy to check that
\[
\zeta_{k}^\varepsilon=-\tfrac{1}{2}\bigl(
\zeta_{k-1}^\varepsilon\bigr)^3+ \tfrac{3}{4}\bigl(
\zeta_{k-1}^\varepsilon\bigr)^2 +\tfrac{3}{4}
\zeta_{k-1}^\varepsilon,
\]
and as before $\zeta_{0}^\varepsilon=1-\frac{\varepsilon}{2}$. Letting
%
%
\begin{equation}
\label{eq-hx-def} h(x)=-\tfrac{1}{2}x^3+\tfrac{3}{4}x^2+
\tfrac{3}{4}x,
\end{equation}
we thus have $\zeta_{k}^\varepsilon=h(\zeta_{k-1}^\varepsilon)$,
and the
proof follows from the easily verifiable facts that
$h$ maps $[0,1]$ to itself with fixed points at $\{0,1/2,1\}$, out of
which $1/2$ is a repelling fixed point since $h'(1/2)=9/8>1$. Hence,
$\zeta_k^\varepsilon\to1$ as long as $\zeta_0^\varepsilon> 1/2$,
which is
indeed the case by the hypothesis $0<\varepsilon< 1$.
\end{pf}
\end{example*}
We note in passing that the analogue of Claim~\ref{clm5Majority} for
noise sensitivity (rather than strong noise sensitivity) is not
possible for any
nondegenerate sequence $(f_n)$, since $\P(f_{n}(\omega^\varepsilon
)=1\mid f_{n}(\omega)=1 )\leq1-g(\varepsilon)$ for $g(\varepsilon
)\gtrsim
\varepsilon$.

\subsection{Quantitative noise sensitivity for cycles at criticality}
In this section we prove the following stronger form of Theorem~\ref
{thmcube-root-cycle}, offering a more detailed examination of the
phase transition for noise sensitivity around the point where the noise
parameter $\varepsilon$ is of order $n^{-1/3}$.

\begin{theorem}\label{thmcube-root-quant}
Fix $0<a<b$, and let $f_n$ be the property that $\cG(n,p)$ with
$p=(1+O(n^{-1/3}))/n$
contains a cycle of length $\ell\in(a n^{1/3}, b n^{1/3})$.
Then $(f_n)$ is nondegenerate, and according to the noise parameter
$\varepsilon(n)$ we have:
\begin{longlist}[(iii)]
\item If $\varepsilon\gg n^{-1/3}$, then
$(f_n)$ is $\NS$ and furthermore \SNSv\ w.r.t. $\varepsilon$.
\item If $\varepsilon\ll n^{-1/3}$, then
$(f_n)$ is $\NSTAB$ w.r.t. $\varepsilon$.
\item If $\varepsilon\asymp n^{-1/3}$, then
$(f_n)$ is neither $\NS$ w.r.t. $\varepsilon$ nor $\NSTAB$
w.r.t. $\varepsilon$.
\end{longlist}
\end{theorem}

\begin{pf}
Let $G\sim\cG(n,p)$, and let $\omega$ denote its edge configuration
(i.e., $\omega_{uv}$ is set to $1$ if the edge $uv$ is present in $G$
and it is\vspace*{1pt} 0 otherwise).
Let $\lambda_1, \lambda_2>0$ be such that $1-\lambda_1 n^{-1/3} \leq
np \leq1 + \lambda_2 n^{-1/3}$ for all $n$ and let $X_\ell$ count the
number of cycles of length $\ell$ in $G$.
Put $\cI= (a n^{1/3}, b n^{1/3})$, and define
\[
X = \sum_{\ell\in\cI} X_\ell= \#\{W \in
\cW_1\dvtx \omega_W\equiv1\}.
\]
As the number of potential cycles notwithstanding automorphisms in $G$
(i.e., the cardinality of $\cW_1$) is $\frac{1}2\binom{n}{\ell
}(\ell
-1)!$, we see that $\E X_\ell\sim(np)^\ell/(2\ell) $ uniformly over
$\ell\in\cI$, and so
%
%
\begin{equation}
\label{eq-X-first-moment} \bigl(1-o(1)\bigr) e^{-\lambda_1 b} \leq
\frac
{\E X}{{1}/2\log(b/a)} \leq
\bigl(1+o(1)\bigr) e^{\lambda_2 b}.
\end{equation}
At this point, the FKG inequality immediately implies that
%
%
\begin{equation}
\label{eq-X-nonden-1} \P(X=0)\geq\prod_{ \ell\in\cI} \bigl(1 -
p^\ell\bigr)^{{1}/2 \binom{n}{\ell}(\ell-1)!} \geq e^{-(1+o(1))\E
X}
\end{equation}
(where the second inequality used the fact that $1-x = e^{-(1+o(1))x}$
as $x\to0$) which is bounded away from $0$, thanks to~\eqref
{eq-X-first-moment}.

Next, we examine $\Var(X)$. For any two cycles $W \neq W'$, let
$\kappa
(W,W')$ count the number of nontrivial connected components in the
intersection of the edges of $W$ and $W'$ (each of which is a simple
path), and define
\[
\zeta_m:= \mathop{\sum_{W,W'\in\cW_1}}_{\kappa(W,W') =m}
\P(\omega_{W}\equiv1, \omega_{W'}\equiv1 )
\]
for each $m\geq1$. With this notation,
\[
\Var(X) \leq\E X + \sum_{m\geq1} \zeta_m,
\]
prompting the task of estimating the $\zeta_m$'s. In what follows, let
$\ell,\ell'$ run over the potential lengths of $W,W'$, respectively,
while $s$ will run over the total number of edges in the intersection
of $W$ and $W'$. We then have
\begin{eqnarray*}
\zeta_m &\leq&\sum_{\ell\in\cI}\sum
_{\ell'\in\cI}\sum_{m\leq s<\ell} \binomm{s} {m-1}
\bigl(2\ell\ell'\bigr)^m \frac{n^{\ell}p^{\ell}}{2\ell}
\frac{n^{\ell'-(s+m)}p^{\ell'-s}}{2\ell'},
\end{eqnarray*}
where the first term accounts for the partitioning of the $s$ total
edges into the $m$ intersection paths (with room to spare), the second
one accounts for selecting the paths within $W$ (starting point and
direction per path) as well as their position within $W'$ and the final
two terms correspond to selecting $W$ and $W'$ with this intersection
pattern. The fact that $np\leq1+\lambda_2 n^{-1/3}$ translates into
having $(np)^{\ell+\ell'-s } < C$ for $C=e^{2b\lambda_2} $, thus
\begin{eqnarray*}
\zeta_m &\leq&\frac{C}{2n} \sum_{\ell}
\sum_{\ell'} \sum_s
\frac{(2\ell\ell' s / n)^{m-1}}{(m-1)!} \leq\frac{C}2 (b-a)^2 b
\frac{(2b^3)^{m-1}}{(m-1)!}
\end{eqnarray*}
and
\[
\sum_{m\geq1} \zeta_m \leq
\frac{C}2 (b-a)^2 b e^{2b^3} = O(1).
\]
In particular we get that $\E[X^2] = O(1)$.


An immediate consequence of Cauchy--Schwarz is that any nonnegative
random variable $X$ satisfies $\P(X>0)\geq(\E X)^2/\E[X^2]$; thus in
particular $\P(X>0)$ is bounded away from 0. Combining this
with~\eqref
{eq-X-nonden-1}, it now follows that $(f_n)$ is nondegenerate.

\begin{remark*}
Using similar moment analysis, one can infer that the limiting
distribution of $X$ is not Poisson; for instance, already $\zeta_1$ is
uniformly bounded away from 0 [as it is apparent that $\zeta_1 \geq
(\frac{1}2-o(1)) (b-a)^2 a $ from the argument above], and consequently
$\Var(X)$ is bounded away from $\E X$ as $n\to\infty$.
\end{remark*}

\begin{itemize}
\item\emph{Noise sensitivity if and only if $\varepsilon\gg n^{-1/3}$}.
The strong noise sensitivity of $(f_n)$ when $\varepsilon\gg
n^{-1/3}$ will be derived from a calculation akin to the second moment
analysis given above, yet this time it will incorporate the noise in
the following prominent way. For any $W\in\cW_1$ of some length $\ell
$, define
\[
\zeta'_m:= \mathop{\sum_{W'\in\cW_1}}_{\kappa(W,W') =m}
\P\bigl(\omega_{W'}^\varepsilon\equiv1\mid\omega_{W}
\equiv1 \bigr).
\]
By the same line of arguments presented above for $\zeta_m$, we have
\begin{eqnarray*}
\zeta'_m &\leq&\sum_{\ell'}
\sum_s \binomm{s} {m-1} \bigl(2\ell
\ell'\bigr)^m \frac{n^{\ell'-(s+m)}p^{\ell'-s}}{2\ell'}\bigl
(1-\varepsilon(1-p)
\bigr)^s
\\
&\leq&\frac{C\ell}n \sum_{\ell'}\sum
_s \frac{(2\ell\ell' s/n)^{m-1}}{(m-1)!} \bigl(1-\varepsilon(1-p)
\bigr)^s,
\end{eqnarray*}
again using the fact that $(np)^{\ell'-s} < C$ for $C=e^{\lambda_2 b}$.
Thanks to the crucial last term, accounting for the probability of
retaining the $s$ edges in the intersection paths, it follows that
\begin{eqnarray*}
\zeta'_m &\leq&\frac{C b (b-a)}{n^{1/3}} \frac{(2b^3)^{m-1}}{(m-1)!}
\sum_s \bigl(1-\varepsilon(1-p)\bigr)^s
\leq\frac{C b (b-a)}{n^{1/3} \varepsilon(1-p)} \frac
{(2b^3)^{m-1}}{(m-1)!},
\end{eqnarray*}
and so
%
%
\begin{equation}
\label{eq-sum-zeta} \sum_{m\geq1}\zeta'_m
\leq\frac{C b(b-a)e^{2b^3}}{n^{1/3} \varepsilon(1-p)} = O \biggl(
\frac
{1}{\varepsilon n^{1/3}} \biggr).
\end{equation}
In\vspace*{1pt} particular, when $\varepsilon\gg n^{-1/3}$ [part~(i)]
we can infer that $\sum_{m\geq1} \zeta'_m = o(1)$.
To deduce that $(f_n)$ is \SNSv\ in this case, argue as follows. Fix in
what follows some $W \in\cW_1$.
Partitioning $\cW_1 = \{W\} \cup\cW_1' \cup\cW_1'' $ where $\cW
_1':=\{
W' \neq W\dvtx \kappa(W,W')>0\}$ (and $\cW_1''$ contains cycles that are
edge-disjoint from $W$,
thus independent) gives
\begin{eqnarray*}
\P\bigl(f_n\bigl(\omega^\varepsilon\bigr)
=1 \mid\omega_W\equiv1\bigr)
&\leq& \P\biggl(\bigcup
_{W' \in\cW_1'} \bigl\{\omega_{W'}^\varepsilon\equiv1
\bigr\} \Big|\omega_W\equiv1 \biggr)
\\
&&{}
+ \P\biggl(\bigcup_{W'' \in\cW_1''} \bigl\{
\omega_{W''}^\varepsilon\equiv1\bigr\} \Big|\omega_W
\equiv1 \biggr)
\\
&&{} + \P\bigl(\omega_{W}^\varepsilon\equiv1 \mid
\omega_W\equiv1 \bigr).
\end{eqnarray*}
By the definition of $\zeta'_m$ and equation~\eqref{eq-sum-zeta} in
the case of $\varepsilon\gg n^{-1/3}$,
\begin{eqnarray*}
\P\biggl(\bigcup_{W' \in\cW_1'} \bigl\{
\omega_{W'}^\varepsilon\equiv1\bigr\} \Big|\omega_W
\equiv1 \biggr) &\leq&
\sum_{m\geq1}
\zeta'_m = o(1),
\end{eqnarray*}
while clearly
\[
\P\biggl(\bigcup_{W'' \in\cW_1''} \bigl\{
\omega_{W''}^\varepsilon\equiv1\bigr\} \Big|\omega_W
\equiv1 \biggr) = \P\biggl(\bigcup_{W'' \in\cW_1''} \{
\omega_{W''}\equiv1\} \biggr) \leq\P(f_n=1)
\]
and
\[
\P\bigl(\omega_{W}^\varepsilon\equiv1 \mid
\omega_W\equiv1 \bigr) = \bigl(1-\varepsilon(1-p)\bigr)^{\ell}
\leq e^{-\varepsilon(1-p)an^{1/3}} = o(1),
\]
again thanks to the assumption that $\varepsilon\gg n^{-1/3}$.
Altogether, this yields
\begin{eqnarray*}
\P\bigl(f_n\bigl(\omega^\varepsilon\bigr)&=&1\mid
\omega_W\equiv1\bigr) \leq\P(f_n=1) +o(1),
\end{eqnarray*}
thus establishing that $(f_n)$ is \SNSv\ when $\varepsilon\gg n^{-1/3}$.

We will now show that $(f_n)$ is not $\NS$ w.r.t. $\varepsilon$ whenever
$\varepsilon= O(n^{-1/3})$, to which end we will appeal to the Fourier
representation described in~Section~\ref{secprelim}.
The first observation, using Lemma~\ref{lemmaPivotalFormula}, is that
the set of pivotals $\cP_n$ satisfies
\[
\E\vert\cP_n\vert= p^{-1} \P(f_n =
1) \E\bigl[ \vert\cP_n\vert\mid f_n=1\bigr]
\leq p^{-1} b n^{1/3},
\]
where the last inequality relied on the fact that given that there
exists some cycle $C_\ell$ with $\ell\in\cI$ in $G$, every pivotal edge
must in particular belong to $C_\ell$, and so there can be at most
$\ell\leq b n^{1/3}$ such edges.
By~\eqref{eFourAndInfluencesGeneralp}, the spectral sample $\cS_n$ satisfies
\[
\E\vert\cS_n\vert= p(1-p)\E\vert\cP_n
\vert\leq b n^{1/3},
\]
which will rule out noise sensitivity for $(f_n)$ w.r.t. $\varepsilon$ by
a standard argument.
As we have established above that $(f_n)$ is nondegenerate, let
$\theta< 1$ be some constant such that $\P(f_n=1) < \theta$ for any
sufficiently large $n$, and set
\[
M = 2b/(1-\theta).
\]
Since $\P(\cS_n = \varnothing) = \P(f_n=1) < \theta$ while $\P
(\vert
\cS_n\vert>M n^{1/3}) \leq(1-\theta)/2$ by Markov's inequality, we
deduce that
\[
\P\bigl( 0 < \vert\cS_n\vert< M n^{1/3} \bigr) > 1
- \theta- \frac{1-\theta}2 = \frac{1-\theta}2,
\]
and in particular this probability is bounded away from 0.
Due to the hypothesis $\varepsilon= O(n^{-1/3})$, we further have
\[
(1-\varepsilon)^{\vert\cS_n\vert}\One_{\{0<\vert\cS_n\vert
< Mn^{1/3}\}} \geq e^{-(1-o(1))\varepsilon M n^{1/3}} \geq c
\]
for some fixed $c>0$, and altogether we obtain that
\[
\liminf_{n\to\infty} \Cov\bigl(f_{n}(
\omega),f_{n}\bigl(\omega^{\varepsilon}\bigr) \bigr)=\liminf
_{n\to\infty} \E\bigl[(1-\varepsilon)^{\vert\cS_n\vert}
\One_{\{\cS_n\neq\varnothing\}} \bigr] > 0;
\]
that is, $(f_n)$ is not $\NS$ w.r.t. $\varepsilon$ in this regime.

\item\emph{Noise stability if and only if $\varepsilon= o(n^{-1/3})$}.
Let $\omega$ be any configuration corresponding to a graph
for which $f_n=1$, where by definition there exists some cycle $W$ of
length $\ell\in(an^{1/3},bn^{1/3})$ such that $\omega_W\equiv1$. Under
the assumption $\varepsilon\ll n^{-1/3}$, we have that $\P(\omega
^\varepsilon
_W \equiv1\mid\omega) \geq1-\varepsilon b n^{1/3} = 1-o(1)$. In other
words, for any $\omega$ such that $f_n(\omega)=1$, we have
$ \P(f_n(\omega^\varepsilon)=1 \mid\omega) = 1-o(1)$,
implying that $(f_n)$ is $\NSTAB$ w.r.t. $\varepsilon$.

To see that $(f_n)$ is not $\NSTAB$ w.r.t. $\varepsilon$ whenever
$\varepsilon
\gtrsim n^{-1/3}$, observe first that if $W$ corresponds to a cycle of
length $\ell\in\cI$, then
\[
\P\bigl(\omega_{W}^\varepsilon\not\equiv1 \mid
\omega_W\equiv1 \bigr) =1- \bigl(1-\varepsilon(1-p)
\bigr)^{\ell} \geq c_0
\]
for some fixed $c_0>0$ which depends on $a$ as well as the implicit
constant in the assumption $\varepsilon\gtrsim n^{-1/3}$.
At the same time, with the same notation as above,
\[
\P\biggl(\bigcap_{W'' \in\cW_1''} \bigl\{
\omega_{W''}^\varepsilon\not\equiv1\bigr\} \Big|
\omega_W\equiv1 \biggr) \geq\P(f_n=0) >
c_1
\]
for some fixed $c_1>0$ thanks to the above established fact that
$(f_n)$ is nondegenerate,
whereas by FKG,
\begin{eqnarray*}
\P\biggl(\bigcap_{W' \in\cW_1'} \bigl\{
\omega_{W'}^\varepsilon\not\equiv1\bigr\} \Big|
\omega_W\equiv1 \biggr)&\geq&\prod_{W'\in\cW_1'}
\P\bigl(\omega_{W'}^\varepsilon\not\equiv1 \mid
\omega_W\equiv1 \bigr)
\\
&\geq& e^{-(1-o(1))\sum_{m\geq1}\zeta'_m} \geq c_2
\end{eqnarray*}
for some fixed $c_2>0$ which depends on $a,b$ and the constant in the
hypothesis $\varepsilon\gtrsim n^{-1/3}$ as specified in~\eqref{eq-sum-zeta}.
Combining the last three inequalities, again by virtue of FKG, we
deduce that
\[
\P\bigl(f_n\bigl(\omega^\varepsilon\bigr)=1 \mid
\omega_W\equiv1 \bigr) \leq1 - c_0 c_1
c_2,
\]
which by equation~\eqref{eq-ns-leq-sns} implies that $\P(f_n(\omega
^\varepsilon)=1\mid f_n(\omega)=1 )$ is bounded away from $1$, precluding
noise stability.
\end{itemize}
This completes the proof.
\end{pf}

%
\begin{remark}
\label{remcrit-no-ns}
One can construct a function which exhibits a phase transition at the
critical window of $\cG(n,p)$, and yet not only is a noise of
$\varepsilon
\gg n^{-1/3}$ (effectively moving $\omega^\varepsilon$ to the subcritical
degenerate regime and then back into the critical window) insufficient
for decorrelating $f_n(\omega),f_n(\omega^\varepsilon)$, neither
does any
fixed $\varepsilon>0$. The following example demonstrates this.

For some constants $0<a<b$ to be determined below, let
$f_n$ the property that the largest component of $G$, denoted by $\cC
_1$, either satisfies $\vert\cC_1\vert> b n^{2/3}$, or
alternatively $a n^{2/3} < \vert\cC_1\vert\leq b n^{2/3}$ while
$G$ further contains a triangle.

Clearly, $\P(f_n=1)=o(1)$ when $G\sim\cG(n,p)$ for $p=(1-\xi)/n$ with
$\xi\gg n^{-1/3}$ as in that case $\vert\cC_1\vert
=o(n^{2/3})$, whereas $\P(f_n=1)=1-o(1)$ when $p=(1+\xi)/n$ for the
same $\xi$ since $\vert\cC_1\vert$ then concentrates around $2
\xi n \gg n^{2/3}$; see, for example,~\cite{Bollobas}, Chapter~6,
and~\cite{JLR}, Chapter~5.\vspace*{1pt}

At $p=(1\pm\xi)/n$ for $\xi=O(n^{-1/3})$ the sequence $(f_n)$ is
nondegenerate. An immediate way to ensure this would be to select $a$
sufficiently small and $b$ sufficiently large. Indeed, it is well known
that $\vert\cC_1\vert/n^{2/3}$ converges in probability to a
nontrivial distribution with full support on $\R_+$, and in particular
for any small $\delta> 0$ we can select $a$ sufficiently small and $b$
sufficiently large so that $\P(a < \vert\cC_1\vert n^{-2/3} < b)
> 1-\delta$. On this event, $f_n$ identifies with the property $g_n$ of
containing a triangle, which is known to be noise stable. In particular,
\[
\P\bigl(f_n\bigl(\omega^\varepsilon\bigr)= f_n(
\omega) \bigr) \geq\P\bigl(g_n\bigl(\omega^\varepsilon
\bigr)=g_n(\omega) \bigr) - 2\delta\geq1 - \delta'
\]
for some $\delta'(\varepsilon,a,b)$ which can be made arbitrarily small
for suitable $\varepsilon,a,b$. This precludes the noise sensitivity of
$f_n$ for any fixed $ \varepsilon>0$, as claimed.

We note in passing that $f_n$ satisfies $\sum_x \hat{f}_n(x)^2
=O(n^{-2/3}) = o(1)$; that is, the BKS criterion for $\NS$ is met, and
nevertheless $(f_n)$ is not $\NS$.
\end{remark}

\subsection{Quantitative noise sensitivity for minimum degree}
Analogously to the previous section, here we prove a stronger version
of Theorem~\ref{thmmin-degree}, which addresses the noise stability
vs. sensitivity at the critical noise level.

\begin{theorem}\label{thmmin-degree-quant}
Let $f_n$ be the property that the minimum degree of $\cG(n,p)$ is at
least $k$ for some fixed $k\geq1$, and suppose $p=p(n)$ is such that
$(f_n)$ is nondegenerate. The following holds depending on the noise
parameter $\varepsilon(n)$:
\begin{longlist}[(iii)]
\item If $\varepsilon\gg\frac{1}{\log n}$,
then $(f_n)$ is $\NS$ and furthermore \SNSn\ w.r.t. $\varepsilon$.\vspace*{2pt}
\item If $\varepsilon\ll\frac{1}{\log n}$,
then $(f_n)$ is $\NSTAB$ w.r.t. $\varepsilon$.\vspace*{2pt}
\item If $\varepsilon\asymp\frac{1}{\log n}$,
then $(f_n)$ is neither $\NS$ w.r.t. $\varepsilon$ nor $\NSTAB$
w.r.t. $\varepsilon$.
\end{longlist}
Moreover, the classification into $\NS$ w.r.t. $\varepsilon$
in~\textup{(i)}, $\NSTAB$ w.r.t. $\varepsilon$ in~\textup{(ii)}
or neither in~\textup{(iii)} holds
for all graph properties listed in Theorem~\ref{thmmin-degree}.
\end{theorem}

\begin{pf}
Let $G\sim\cG(n,p)$, and let $\omega$ denote its edge configuration.
Fix $k\ge1$, and let $D_n$ be the graphs (or corresponding
configurations $\omega$) with minimum degree at least $k$, so that
$f_n(\omega) = \One_{\{\omega\in D_n\}}$. The assumption that $(f_n)$
is nondegenerate is well known
(see, e.g., \cite{Bollobas,JLR}) to correspond to
%
%
\begin{equation}
\label{eq-min-degree-p} p = \frac{\log n + (k-1)\log\log n + O(1)}{n}.
\end{equation}

Consider first the range $\frac{1}{\log n} \ll\varepsilon< 1$. In this
regime, we wish to compare $\P(\omega^\varepsilon\in D_n^c \mid
\omega_W
\equiv0) $ to $\P(\omega\in D_n^c)$ for any 0-witness $W$ for $D_n$.
Clearly, such a 0-witness $W$ is precisely a set of $n-k$ edges
incident to a
vertex. Denoting the vertices by $v_1,v_2,\ldots,v_n$, assume without
loss of generality
that this $W$ consists of the edges $\{ v_1 v_i\dvtx i=2,\ldots
,n-k+1\}$.
By the symmetry of witnesses, it is enough to show that for each
$\varepsilon>0$,
%
%
\begin{equation}
\label{eqnNeededInequality} \liminf_{n\to\infty} \P\bigl(\omega
^{\varepsilon}
\in D_n\mid\omega_{W} \equiv0 \bigr)-\P(\omega\in
D_n ) \ge0.
\end{equation}

Let $A_n$ be the event that the induced subgraph on the vertices $\{
v_2,\ldots,v_n\}$
has minimum degree at least $k$.
We claim that
%
%
\begin{equation}
\label{eqnProjectedTon-1} \liminf_{n\to\infty}\P(\omega\in A_n)-
\P(\omega\in D_n)\ge0.
\end{equation}
(The limit is in fact 0, but this will not be needed.) It suffices to
show that
\[
\lim_{n\to\infty}\P\bigl(\omega\in A^c_n
\cap D_n\bigr)=0.
\]
Any graph in $A^c_n\cap D_n$ has some vertex $v_i$ with $2\leq i \leq
n$ such that
the degree of~$v_i$ is precisely $k$, and $v_1 v_i$ is an edge. By a
union bound,
the probability that $\omega$ satisfies the latter is at most
\[
(n-1) \binomm{n-2} {k-1}p^k (1-p)^{n-1-k}\le
(np)^k e^{-p(n-1-k)} \lesssim\frac{\log n}{n} = o(1),
\]
having plugged in the expression for $p$ from~\eqref{eq-min-degree-p}.
This establishes~\eqref{eqnProjectedTon-1}.

Next, let $B_n$ be the set of graphs where the degree of $v_1$ is at
least $ k$.
We claim that
%
%
\begin{equation}
\label{eqnFirstVertexFine} \lim_{n\to\infty} \P\bigl(\omega
^{\varepsilon
}\in
B_n\mid\omega_W \equiv0 \bigr)=1.
\end{equation}
Indeed, if $C_n$ is the set of graphs where $v_1$ is isolated, then $\P
(\omega\in\cdot\mid\omega_W\equiv0)$ stochastically dominates $\P
(\omega\in\cdot\mid\omega\in C_n)$ where $\P(\omega\in\cdot
\mid A)$
denotes the conditional distribution of $\omega$ conditioned on $A$.
Thus, as $B_n$ is increasing, by FKG we have
%
%
\begin{equation}
\label{eq-min-deg-eps-dependency}
\qquad\P\bigl(\omega^{\varepsilon}\in B_n\mid
\omega_W \equiv0 \bigr) \geq\P\bigl(\omega^{\varepsilon}\in
B_n\mid\omega\in C_n \bigr) = \P\bigl(\Bin(n-1,
\varepsilon p) \geq k \bigr). 
\end{equation}
Since $p\sim\frac{\log n}n$ and $\varepsilon\gg\frac{1}{\log n}$, the
above binomial variable concentrates on $(n-1)\varepsilon p \gg k$;
hence the last expression is $1-o(1)$.
This demonstrates~\eqref{eqnFirstVertexFine}.

To put it all together, observe that
\begin{eqnarray*}
\P\bigl(\omega^{\varepsilon}\in D_n\mid\omega_W
\equiv0 \bigr) &\geq& \P\bigl(\omega^{\varepsilon}\in A_n\cap
B_n\mid\omega_W \equiv0 \bigr)
\\
& =& \P\bigl(\omega^{\varepsilon}\in A_n\mid
\omega_W \equiv0 \bigr)\P\bigl(\omega^\varepsilon\in
B_n\mid\omega_W\equiv0 \bigr),
\end{eqnarray*}
since the events $A_n$ and $B_n$ are (conditionally) independent.
Plugging in~\eqref{eqnFirstVertexFine} and using the independence of
$\{\omega^{\varepsilon}\in A_n\}$ and $\{\omega_W \equiv0 \}$, we
conclude that
\[
\P\bigl(\omega^{\varepsilon}\in D_n\mid\omega_W
\equiv0 \bigr) \geq\P\bigl(\omega^{\varepsilon}\in A_n \bigr) -
o(1),
\]
and the required inequality~\eqref{eqnNeededInequality} now follows
from~\eqref{eqnProjectedTon-1} and completes the proof of part~(i).

For part~(ii) consider any $\omega\in D_n^c$,
whereby the corresponding graph $G$ contains some vertex $v_i$ of
degree less than $k$. Since $\varepsilon= o(1/\log n)$, the probability
that the degree of $v_i$ increases due to the noise is at most
$(n-1)\varepsilon p = o(1)$, and so
$\P(\omega^\varepsilon\in D_n^c \mid\omega) = 1-o(1)$.
Translating this in terms of $f_n$, for any $\omega$ such that
$f_n(\omega)=0$ we have
$ \P(f_n(\omega^\varepsilon)=0 \mid\omega) = 1-o(1)$,
which establishes noise stability w.r.t. $\varepsilon$.

We next proceed to part~(iii), addressing the
critical regime of $\varepsilon\asymp\frac{1}{\log n}$. To show $(f_n)$
is not $\NSTAB$ w.r.t. $\varepsilon$, note first that the binomial
variable in the right-hand side of~\eqref{eq-min-deg-eps-dependency} is
now approximately Poisson with mean bounded away from $0$ and $\infty$,
implying (by the same line of arguments as above) that
\[
\P\bigl(\omega^{\varepsilon}\in D_n\mid\omega_W
\equiv0 \bigr) \geq\delta\P(\omega\in D_n)
\]
for some fixed $\delta>0$ and all $n$, or equivalently,
\[
\P\bigl(\omega^{\varepsilon}\in D_n^c\mid
\omega_W \equiv0 \bigr) \leq1- \delta\P(\omega\in
D_n).
\]
Appealing to equation~\eqref{eq-ns-leq-sns} from the proof of
Lemma~\ref
{lemmaSNStoNS}, and using the symmetry of 0-witnesses, we now deduce that
\[
\P\bigl(f_n\bigl(\omega^{\varepsilon}\bigr)=0 \mid
f_n(\omega)=0 \bigr) \leq1- \delta\P(f_n=1),
\]
which precludes noise stability w.r.t. $\varepsilon$ as $(f_n)$ is
nondegenerate.\vspace*{1.5pt}

To rule out noise sensitivity for $\varepsilon\asymp\frac{1}{\log n}$,
as in the proof of Theorem~\ref{thmcube-root-quant} we appeal to the
Fourier representation of $f_n(\omega^\varepsilon)$.
For any $\omega$ such that $f_n(\omega)=0$, an edge $uv$ can only be
pivotal if every $w\neq u,v$ has degree at least $k$ in $\omega$.
Moreover, if both $u,v$ have degree $k-1$ in $\omega$, then this would
be the unique pivotal edge, and otherwise $\vert\cP_n\vert=n-k$.
In particular, using~\eqref{eFourAndInfluencesGeneralp} and Lemma~\ref
{lemmaPivotalFormula}, we see that
\[
\E\vert\cS_n\vert=p(1-p) \E\vert\cP_n
\vert= p \P(f_n=0) \E\bigl[\vert\cP_n\vert
\mid f_n=0\bigr] \leq\bigl(1+o(1)\bigr) \log n.
\]
As $(f_n)$ is nondegenerate by hypothesis, let $\theta< 1$ be some
constant such that $\P(f_n=1) < \theta$ for large enough $n$, and set $
M = 2/(1-\theta)$.
Since the spectral sample $\cS_n$ satisfies $\P(\cS_n = \varnothing
) = \P
(f_n=1)$, Markov's inequality implies that
\[
\P\bigl( 0 < \vert\cS_n\vert< M \log n \bigr) > 1 - \theta
- \frac{1-\theta}2 - o(1) = \frac{1-\theta}2 - o(1).
\]
Consequently, when $\varepsilon= O(1/\log n)$, there exists some $c>0$
such that
\[
(1-\varepsilon)^{\vert\cS_n\vert}\One_{\{0<\vert\cS_n\vert
< M \log n\}} \geq e^{-(1-o(1))\varepsilon M \log n} \geq c > 0,
\]
and so
\[
\liminf_{n\to\infty} \Cov\bigl(f_{n}(
\omega),f_{n}\bigl(\omega^{\varepsilon}\bigr) \bigr)=\liminf
_{n\to\infty} \E\bigl[(1-\varepsilon)^{\vert\cS_n\vert}
\One_{\{\cS_n\neq\varnothing\}} \bigr] > 0;
\]
that is, $(f_n)$ is not $\NS$ w.r.t. $\varepsilon$ in this regime.

Finally, it remains to extend the classification of either $\NS$ or
$\NSTAB$ w.r.t. $\varepsilon$ to the graph properties listed in
Theorem~\ref{thmmin-degree}. To this end, recall the well-known facts
(see~\cite{Bollobas,JLR,BHKL})
that each such property $(g_n)$ is asymptotically equal to the property
$(f_n)$ of having minimum degree at least $k$ (for an appropriate $k$),
in the sense that $\lim_{n\to\infty}\P(f_n\neq g_n)= 0$. It is
elementary that if $(f_n)$ is noise sensitive
(noise stable) and $(g_n)$ is asymptotically equal to $(f_n)$, then
$(g_n)$ is noise sensitive
(noise stable), since
\[
\bigl\vert\E\bigl[f_n\bigl(\omega^\varepsilon
\bigr)f_n(\omega)\bigr]-\E\bigl[g_n\bigl(
\omega^\varepsilon\bigr)g_n(\omega)\bigr]\bigr\vert\le2
\P(f_n\neq g_n),
\]
thus translating the quantitative statements on $(f_n)$ to $(g_n)$, as required.
\end{pf}

%
\begin{remark}
As an alternative way to obtain noise sensitivity for $\cG(n,p)$ having
minimum degree at least $k$, 
one could appeal to~\cite{ScSt}, Theorem~1.8, and present a randomized
algorithm
for this event whose probability of querying any given edge tends to 0.
This would imply a
quantitative noise sensitivity result, albeit weaker than the sharp one
obtained above.
\end{remark}

\section{Noise sensitivity of witness-transitive functions} \label
{secwit-dist}
Let $f$ be a monotone Boolean function on a domain $\Omega$. We say
that $f$ is 1-\emph{witness-transitive} if the set of automorphisms of
$f$ (the set of permutations $\pi$ on $\Omega$ under which $f$ is
invariant, i.e., $f \equiv f\circ\pi$) is such that for any two
witnesses $W,W'\in\cW_1(f)$ there exists an automorphism of $f$ mapping
$W$ to $W'$. That is to say, any two 1-witnesses for $f$ are equivalent.

For instance, the classical examples for noise sensitive functions
which were mentioned in the \hyperref[secIntro]{Introduction}, tribes and recursive
majority, are both 1-witness-transitive, as is the property of
containing an unlabeled copy of a certain graph $H$ in a random graph
$G\sim\cG(n,p)$.

\subsection{A Poissonization tool for strong noise sensitivity}
Our goal in this section is to prove a sufficient condition for strong
noise sensitivity of 1-witness-transitive functions.
This condition will be in the form of a Poisson approximation of the
total number of occurring 1-witnesses, as stated next.

\begin{proposition}\label{proppoisson-WD}
Let $(f_n)$ be a sequence of 1-witness-transitive monotone Boolean
functions. Let $W_\star=W_\star(n)$ be a canonical 1-witness for
$f_n$, and
suppose that $(1-p_n)\vert W_\star\vert\to\infty$ with\vspace
*{1pt} $n$.
Let
$X_n = \sum_{W\in\cW_1(f_n)} \One_{\{\omega_W\equiv1\}}$ count the
occurring 1-witnesses, and assume
that for some $\lambda\in\R_+$, we have
%
%
\begin{eqnarray}
\label{eq-poisson-hypo-moments} \lim_{n\to\infty}\E[X_n]&=&\lambda
\quad
\mbox{and}\quad\lim_{n\to\infty}\Var(X_n)=\lambda,
\\
\label{eq-poisson-hypo-E-minus} \lim_{n\to\infty} \E[ X_n \mid
\omega_{W_\star} \equiv0 ] &=& \lambda.
\end{eqnarray}
Then $X_n\stackrel{\mathrm{d}}\to\Po(\lambda)$ as $n\to\infty$, and
$(f_n)$ is $\NS$ and moreover \SNSv. Furthermore,
quantitative $\NS$ (as well as \SNSv) holds w.r.t. $\varepsilon(n)$
if and
only if
%
%
\begin{equation}
\label{eq-poisson-hypo-epsilon} \varepsilon\gg\bigl[(1-p_n)\vert
W_\star
\vert\bigr]^{-1}.
\end{equation}
\end{proposition}

\begin{pf}
The fact that the $X_n$ converges in distribution to a Poisson random
variable under the given assumptions follows from a standard
application of the Chen--Stein method; see, for example,~\cite
{AGG}, Theorem~1, and~\cite{JLR}, Theorem~6.24. Indeed, writing $I_W =
\One_{\{\omega_W\equiv1\}}$ for $W\in\cW_1$ we see that $\P(I_W) =
p^{\vert W\vert} = o(1)$ thanks to the assumption $(1-p)
\vert W_\star\vert\to\infty$. As these indicators are
positively related by FKG, we can invoke a simplified form of the
Chen--Stein method (see~\cite{JLR}, Theorem~6.24), at which point the
assumptions of~\eqref{eq-poisson-hypo-moments} imply that
\[
\bigl\llVert X_n - \Po(\lambda)\bigr\rrVert_\tv\leq
\frac{\Var(X_n)}{\E[X_n]}-1+2\max_{W\in\cW_1} \P(I_W) = o(1).
\]

Linking the above to strong noise sensitivity will be achieved by the
next key definition, which we phrase for general monotone Boolean
functions (not necessarily witness-transitive) as it may be of
independent interest. The proof of Proposition~\ref{proppoisson-WD}
will be continued after this detour.

\begin{definition} \label{def-1-witness-disjoint}
A sequence $(f_{n})$ of monotone increasing Boolean functions
is said to be 1-\emph{witness-disjoint} if
\[
\lim_{n\to\infty} \max_{W\in\cW_1} \P
\biggl(\mathop{\bigcup_{W' \in\cW_1\setminus\{W\}}}_{W'\cap W
\neq
\varnothing} \{
\omega_{W'}\equiv1\} \Big|\omega_W\equiv1 \biggr) =0.
\]
\end{definition}


Note that the above condition would trivially hold if every pair of
distinct \mbox{1-}witnesses were disjoint (as is the case, e.g., for
the tribes function, where the 1-witnesses are full blocks). In a
sense, Definition~\ref{def-1-witness-disjoint} provides an
approximation to such a situation, which, as we show next, is powerful
enough to imply (quantitative) strong noise sensitivity.


%
\begin{lemma}\label{lemWDtoSNS}
Let $(f_{n})$ be a sequence of monotone Boolean functions that is
1-witness-disjoint. Let $\varepsilon(n)$ be such that $ \varepsilon(1-p_n)
\ell_n\to\infty$ with $n$, where $\ell_n$ is the minimum size of a
1-witness for $f_n$. Then $(f_n)$ is \SNSv\ w.r.t. $\varepsilon$.
\end{lemma}

\begin{pf}
Thanks to our assumption on $\varepsilon$ we have that for any
1-witness $W$,
\[
\P\bigl(\omega^\varepsilon_W\equiv1 \mid
\omega_W\equiv1 \bigr) = \bigl(1-\varepsilon(1-p)\bigr)^{\vert
W\vert}
\leq e^{-\varepsilon(1-p_n)\ell_n} = o(1),
\]
and therefore
%
%
\begin{eqnarray}\label{eq-sns-bound-1}
\P\bigl(f_n\bigl(\omega^{\varepsilon}\bigr)=1\mid
\omega_W\equiv1\bigr)&=& \P\biggl(\bigcup
_{
W' \in\cW_1}\bigl\{ \omega^\varepsilon_{W'}\equiv1
\bigr\} \Big|\omega_W\equiv1 \biggr)
\nonumber
\nonumber\\[-8pt]\\[-8pt]\nonumber
&\le&\P\biggl(\bigcup_{
W' \in\cW_1\setminus\{W\}} \bigl\{
\omega^\varepsilon_{W'}\equiv1\bigr\} \Big|\omega_W
\equiv1 \biggr) +o(1).
\end{eqnarray}
Define the events $A_n$ and $B_n$ by
\begin{eqnarray*}
A_n &=& \mathop{\bigcup_{W' \in\cW_1}}_{W' \cap W = \varnothing}
\bigl\{ \omega^\varepsilon_{W'}\equiv1\bigr\},\qquad
B_n = \mathop{\bigcup_{W' \in\cW_1\setminus\{W\}}}_{W' \cap W \neq
\varnothing}
\bigl\{ \omega^\varepsilon_{W'}\equiv1\bigr\}.
\end{eqnarray*}
Of course,
$\P(A_n \mid\omega_W\equiv1) \leq\P(f_n=1)$ as the events $A_n$ and
$\{\omega_W\equiv1\}$ are mutually independent,
and together with~\eqref{eq-sns-bound-1} this yields
%
%
\begin{equation}
\label{eq-sns-bound-2} \P\bigl(f_n\bigl(\omega^{\varepsilon}\bigr
)=1\mid
\omega_W\equiv1\bigr) - \P(f_n=1 ) \leq\P
(B_n\mid\omega_W\equiv1 ) + o(1).
\end{equation}
Next, since the distribution of $\omega^{\varepsilon}$ conditioned on
$\omega_W \equiv1$ is
stochastically dominated by the distribution of $\omega$ conditioned on
$\omega_W\equiv1$,
\[
\P(B_n\mid\omega_W\equiv1 ) \leq\P\biggl(\mathop{
\bigcup_{W' \in\cW_1\setminus\{W\}}}_{W \cap W' \neq\varnothing}
\{
\omega_{W'}\equiv1\} \Big|\omega_W\equiv1 \biggr).
\]
Now take a supremum over $W\in\cW_1$, under which the final expression
goes to 0 by Definition~\ref{def-1-witness-disjoint}. Combined
with~\eqref{eq-sns-bound-2}, this completes the proof.
\end{pf}

Returning to the proof of Proposition~\ref{proppoisson-WD}, we claim
that under the hypotheses $\E X_n \to\lambda$ and $\E[X_n\mid\omega
_{W_\star}\equiv0] \to\lambda$ given there, the extra assumption
\mbox{$\Var
(X_n)\to\lambda$} in~\eqref{eq-poisson-hypo-moments} is equivalent
to having
%
%
\begin{equation}
\label{eq-Hn-wit-disjoint} \lim_{n\to\infty} \mathop{\sum
_{W \in\cW_1\setminus\{W_\star\}}}_{W\cap W_\star\neq\varnothing
}\P(
\omega_{W}\equiv1\mid
\omega_{W_\star}\equiv1 ) =0.
\end{equation}
As per Definition~\ref{def-1-witness-disjoint}, this would imply
(thanks to the witness-transitivity) that $(f_n)$ is
1-witness-disjoint, and in light of Lemma~\ref{lemWDtoSNS} we will
thereafter arrive at strong noise sensitivity w.r.t. $\varepsilon$
assuming $\varepsilon\gg[(1-p_n)\vert W_\star\vert]^{-1}$.
Indeed, this equivalence is seen by expanding $\E X_n^2=\E X_n + \Gamma
+ \Delta$ where
\begin{eqnarray*}
\Gamma&=& \mathop{\sum_{W, W' \in\cW_1}}_{W'\cap W = \varnothing}
\P(
\omega_W\equiv1, \omega_{W'}\equiv1 ),\qquad\Delta
= \mathop{\sum_{W\neq W' \in\cW_1}}_{W'\cap W \neq\varnothing} \P(
\omega_W\equiv1, \omega_{W'}\equiv1 ).
\end{eqnarray*}
The expression for $\Gamma$, which is clearly at most $(\E X_n)^2$, can
be rewritten by virtue of the
independence of $W,W'$ and the witness-transitivity as
\[
\sum_{W\in\cW_1}\P(\omega_W\equiv1)
\mathop{\sum_{W'\in\cW_1}}_{W\cap W' = \varnothing}\P(
\omega_{W'}\equiv1) = \E[X_n] \E[X_n\mid
\omega_{W_\star}\equiv0],
\]
which is at least $(1-o(1))\lambda^2$ by the aforementioned hypotheses.
At this point, $\Var(X_n)\to\lambda$ if and only if $\Delta\to0$, and
yet by the witness-transitivity,
\[
\Delta= \E[X_n] \mathop{\sum_{W\in\cW_1 \setminus\{W_\star\}
}}_{W\cap
W_\star\neq\varnothing}
\P(\omega_{W}\equiv1\mid\omega_{W_\star} \equiv1 ).
\]
This completes the argument for \SNSv\ whenever $\varepsilon\gg
[(1-p_n)\vert W_\star\vert]^{-1}$.

In the regime $\varepsilon\lesssim[(1-p_n)\vert W_\star\vert
]^{-1}$, the sequence $(f_n)$ will not be $\NS$, by the same Fourier
argument given in the previous section: as before, $\E[\vert\cP
_n\vert\mid f_n = 1] \leq\vert W_\star\vert$ since we can
take an arbitrary witness $W$ that occurs in a configuration for which
$f_n=1$ and note that every pivotal edge must then belong to $W$. It
then follows that $\E\vert\cS_n\vert\leq(1-p_n)\vert W_\star
\vert$,
thus for $\varepsilon\lesssim[(1-p_n)\vert W_\star\vert]^{-1}$ we
have $\liminf_{n\to\infty}\Cov(f_n(\omega),f_n(\omega^\varepsilon
))>0$ due
to the Fourier levels $0<\vert\cS_n\vert< M (1-p_n)\vert
W_\star\vert$ for a suitable constant $M>0$.
\end{pf}

\begin{example*}[(Tribes)]
We have seen in the previous section that the tribes function is \SNSv\
by a direct analysis of $\P(f_n(\omega^\varepsilon)\mid\omega
_W\equiv1)
- \P(f_n=1)$. We will now derive this fact via an immediate application
of Proposition~\ref{proppoisson-WD}. Let $m = \log_2n - \log_2\log_2
n$ denote the block size in $f_n$ (as usual, divisibility issues can be
solved by ignoring one exceptional block; we omit floors and ceilings
for brevity), and note that a canonical 1-witness $W_\star$ consists of
a full block, and so $(1-p_n)\vert W_\star\vert\asymp m \to\infty
$. Moreover, $X_n$ is simply a $\Bin(n/m, 2^{-m})$ random variable.
Thus both $\E[X_n] \to1$ and $\Var(X_n)\to1$ as $n\to\infty$, while
under the conditioning $\omega_{W_\star}\equiv0$, the variable $X_n$
becomes a $\Bin(n/m-1,2^{-m})$ variable, whose mean again converges to
$1$ as $n\to\infty$. The conditions of Proposition~\ref
{proppoisson-WD} are thus met, yielding that $(f_n)$ is \SNSv.
Furthermore, it is such if and only if $\varepsilon\gg1/m$ while it is
not $\NS$ for $\varepsilon=O(1/m)$.
\end{example*}

%
\begin{remark}\label{rem-poisson-WD-no-limits}
It is easily seen from the proof of the above proposition that in order
to conclude (quantitative) strong noise sensitivity without making any
claim on the limiting distribution of $X_n$, conditions~\eqref
{eq-poisson-hypo-moments} and~\eqref{eq-poisson-hypo-E-minus} may be
replaced by
%
%
\begin{eqnarray}
\label{eq-poisson-weak-hypo-E} &\displaystyle0 < \liminf_{n\to
\infty}\E[X_n] \leq
\limsup_{n\to\infty} \E[X_n] < \infty,&
\\
\label{eq-poisson-weak-hypo-var} &\displaystyle\lim_{n\to\infty}
\bigl\vert\Var(X_n)
- \E[X_n]\bigr\vert= 0,&
\\
\label{eq-poisson-weak-hypo-E-minus} &\displaystyle\lim_{n\to\infty
} \bigl\vert\E[X_n]
- \E[ X_n \mid\omega_{W_\star} \equiv0 ]\bigr\vert= 0.&
\end{eqnarray}
Under these assumptions, $(f_n)$ is nondegenerate thanks to FKG
[bounding $\P(X = 0)$ away from 0] and Cauchy--Schwarz [bounding $\P
(X>0)$ away from 0] as in the proof of Theorem~\ref
{thmcube-root-quant}. Following the proof of Proposition~\ref
{proppoisson-WD} we see that, as $\E[X_n]=O(1)$, conditions~\eqref
{eq-poisson-weak-hypo-var} and~\eqref{eq-poisson-weak-hypo-E-minus}
yield $\Delta\to0$, from which point the original argument completes
the proof.
\end{remark}

As an immediate corollary of the results proved above, we get the
following sufficient condition for strong noise sensitivity of
containing an unlabeled copy of a graph in the Erd\H{o}s--R\'enyi
random graph.

\begin{corollary}\label{corsubgraph-WD}
Let $G\sim\cG(n,p)$, and let $H_n$ be a graph with $k\ll\sqrt{n}$
vertices and $\ell\gg1/(1-p)$ edges. Let $f_n=\One_{\{X_n>0\}}$ where
$X_n$ counts the number of unlabeled copies of $H_n$ in $G$, and
suppose that 
\begin{eqnarray*}
&\displaystyle0<\liminf_{n\to\infty} \E[X_n] \leq
\limsup_{n\to\infty} \E[X_n] <\infty,&
\\
&\displaystyle\lim_{n\to\infty} \bigl\vert\Var(X_n)- \E
[X_n]\bigr\vert= 0.&
\end{eqnarray*}
%
Then $(f_n)$ is $\NS$ and moreover \SNSv. Furthermore, quantitative\break 
\SNSv\ holds
if $\varepsilon\gg[(1-p)\ell]^{-1}$, and otherwise $(f_n)$ is not
$\NS
$ w.r.t. $\varepsilon$.
\end{corollary}

\begin{pf}
Appealing to Proposition~\ref{proppoisson-WD}, with the canonical
witness $W_\star$ being a copy of $H_n$, we see that~\eqref
{eq-poisson-weak-hypo-E}, \eqref{eq-poisson-weak-hypo-var} and the fact
that $(1-p_n)\vert W_\star\vert\to\infty$ are explicitly
assumed. For~\eqref{eq-poisson-weak-hypo-E-minus}, the final condition
in Remark~\ref{rem-poisson-WD-no-limits}, note that $\E[X_n] = \binom
{n}k p^\ell k!/ \aut(H_n)$ where $\aut(H_n)$ is the size of the
automorphism group of $H_n$, while $\E[X_n \mid\omega_{W_\star
}\equiv
0] \geq\binom{n-k}k p^\ell k!/ \aut(H_n) \sim\E[X_n]$ thanks to the
hypothesis that $k\ll\sqrt{n}$, as desired.
\end{pf}


\subsection{Noise sensitivity for cliques}
This section is devoted to the noise sensitivity of cliques of any size
$1\ll k_n=n^{o(1)}$ in the random\vspace*{1pt} graph $\cG(n,p)$,
corresponding to
the maximum cliques for $n^{-o(1)} \leq p \leq1-n^{-o(1)}$.
\begin{pf*}{Proof of Theorem~\ref{thmclique}}
The statement of the theorem will follow from Corollary~\ref
{corsubgraph-WD} via the standard second moment analysis which implies
the 2-point concentration of the clique number $k_n$ of $\cG(n,1/2)$,
generalized to the case of $1 \ll k_n = n^{o(1)}$. An outline of this
second moment calculation for $p=1/2$ is given in~\cite{AS,Bollobas},
and here we provide the full details for the sake of completeness.

Let $X_{k}=X_k(n)$ count the number of cliques of size $k=k_n$ in
$G\sim\cG(n,p)$, and
note that
$ \E X_{k} = \binom{n}k p^{\binom{k}2}$
can be assumed to be bounded away from 0, as otherwise $\P
(X_k=0)=1-o(1)$ and so the sequences $k_n,p_n$ would correspond to a
degenerate sequence $(f_n)$ countering the hypothesis of the theorem.

In order to estimate the variance of $X_k$, as usual write $\Var(X_k)
\leq\E X_k + \Delta$ for $\Delta= \sum_{H_1,H_2} \P(H_1\subset G,
H_2\subset G)$, where the summation runs over all pairs of potential
$k$-cliques $H_1\neq H_2$ that have some edges in common.
We claim that the required result would follow from showing that
%
%
\begin{equation}
\label{eq-Delta-o-mean-square} \Delta= o \bigl( (\E X_k)^2 \bigr).
\end{equation}
Indeed, suppose that $\E X_k \to\infty$ with $n$. In this case~\eqref
{eq-Delta-o-mean-square} implies that $\Var(X_k)\ll(\E X_k)^2$. Thus
by Chebyshev's inequality, $X_k$ concentrates about its mean and in
particular $\P(X_k>0)=1-o(1)$, contradicting the hypothesis that
$(f_n)$ is nondegenerate. We thus have that $\E X_k $ is bounded away
from $0$ and $\infty$ for any sufficiently large $n$, and a closer look
at $\E X_k\sim( n p^{(k-1)/2} )^k/k!$ reveals that this can only occur if
%
%
\begin{equation}
\label{eq-pc-clique} p = n^{-(2+o(1))/k}.
\end{equation}
Hence, either $k=O(\log n)$, in which case $p$ is bounded away from $1$
and in particular the number of edges $\ell=\binom{k}2$ satisfies
$\ell
\gg1/(1-p)$, or we have $k \gg\log n$, and then $(1-p)^{-1} =
O(k/\log n) = o(k^2)$, again satisfying the condition $\ell\gg
1/(1-p)$ in Corollary~\ref{corsubgraph-WD}. Finally, it follows
from~\eqref{eq-Delta-o-mean-square} that $\vert\E[X_k] - \Var
(X_k)\vert\to0$ and the mentioned corollary now provides the
required statement on the strong noise sensitivity of $(f_n)$.
Furthermore, we obtain that quantitative (strong) noise sensitivity
holds if and only if $\varepsilon\gg[(1-p)k^2]^{-1}$.

A classical fact worth reiterating is that for $p $ as given in~\eqref
{eq-pc-clique}, and writing $\psi_j=\E[ X_{j+1}]/\E[X_j] $, one has
$\psi_j = p^j (n-j)/(j+1)$. Thus the map $j\mapsto\E X_j$ (starting at
$\E X_1=n$) is unimodal, and for $j\sim k$ it satisfies that $\psi_j =
n^{-1+o(1)}$. By the discussion above, this yields the 2-point
concentration of the clique number, and moreover a 1-point
concentration except for those rare values of $n$ when, for example,
the first $\E X_j$ to drop below $1$ (say) is still bounded away from
0. These are precisely the nondegenerate cases.

To obtain~\eqref{eq-Delta-o-mean-square}, one breaks $\Delta$ down into
$\Delta= \sum_{i=2}^{k-1}\Delta_i$ according to $i$, the number of
common vertices between $H_1,H_2$ (at least 2 to accommodate a common
edge and less than $k$ to keep the cliques distinct), obtaining that
\[
\Delta_i = \binomm{n}k \binomm{k}i \binomm{n-k} {k-i} p^{2\binom
{k}2 -
\binom{i}2}.
\]
Fix any arbitrary $0<\delta<\frac{1}2$, and let
\begin{eqnarray*}
\alpha&:=& (1+\delta)\frac{\log n}{\log(1/p)},\qquad\beta:=
(2-\delta)
\frac{\log(n/k^2)}{\log(1/p)},
\end{eqnarray*}
noting that $\alpha<\beta$ for large enough $n$ since $k=n^{o(1)}$. It
is now easy to see that for any $i \leq\beta$ we have
\[
\frac{\Delta_i}{(\E X_k)^2} = \frac{\binom{k}i \binom
{n-k}{k-i}}{\binom
{n}k p^{\binom{i}2} } \leq\frac{1+o(1)}{i!} \biggl[
\frac{k^2}{n p^{(i-1)/2}} \biggr]^i \leq\frac{1+o(1)}{i!} \biggl(
\frac{k^2}n \biggr)^{\delta i/2},
\]
where the first inequality holds for $k\ll\sqrt{n}$ and the second one
for $i\leq\beta$.
It then follows that
\[
\sum_{2\leq i \leq\beta} \frac{\Delta_i}{(\E X_k)^2} \leq
n^{-\delta
+ o(1)}
= o(1),
\]
and we now proceed to handle the remaining $\Delta_i$'s (with some
overlap). Since $\E X_k$ is bounded away from 0, we see that for any
$\alpha\leq i < k$,
\[
\frac{\Delta_i}{(\E X_k)^2} \lesssim\frac{\Delta_i}{\E X_k} =
\binomm
{k}i \binomm{n-k} {k-i}
p^{\binom{k}2 - \binom{i}2} \leq\frac{ (k(n-k)p^i )^{k-i}}{ ((k-i)!
)^2} \leq\bigl(k n^{-\delta}
\bigr)^{k-i},
\]
with the last inequality stemming from the fact that $i\geq\alpha$. In
particular,
\[
\sum_{\alpha\leq i \leq k-1} \frac{\Delta_i}{(\E X_k)^2} \leq
n^{-\delta+ o(1)}
= o(1),
\]
and as $\alpha< \beta$ this establishes~\eqref
{eq-Delta-o-mean-square}, completing the proof.
\end{pf*}

In the special case where the sequence of probabilities $p(n)$ is such
that $\E[X_k]\to\lambda$ for some fixed $\lambda>0$ [i.e., $\binom
{n}kp^{\binom{k}2}$ converges], the above proof further gives (via the
Chen--Stein method, as in the proof of Proposition~\ref
{proppoisson-WD}) that $X_k\stackrel{\mathrm{d}}\to\Po(\lambda)$.
However, a Poisson limit for the number of copies of a graph is not a
necessary condition for \SNSv, as the next remark shows.

%
\begin{remark}[(Disjoint union of two cliques)]
Consider the property $f_n$ of containing a disjoint union of two
cliques $K_k\cup K_k$ when the clique size $1\ll k\ll n^{o(1)}$ is
exactly such that the probability of witnessing a single such clique in
$G\sim\cG(n,p)$ is nondegenerate. We claim that containing this graph,
which we note is balanced but not strictly balanced, is \SNSv\ despite
the fact that the corresponding number of copies of this graph is not
asymptotically Poisson, nor is this property 1-witness-disjoint.
Indeed, one easily sees that the condition in Definition~\ref
{def-1-witness-disjoint} fails since upon conditioning on two
disjoint cliques $H'$ and $H''$ (which together form a 1-witness for
$f_n$), there exists a third clique $\tilde H$, disjoint from $H'$ and
$H''$, with probability bounded away from 0 (in which case
$\tilde{H}\cup H'$, e.g., would be a 1-witness nontrivially
intersecting $H'\cup H''$).

In order to establish \SNSv\ for this property, we modify the second
moment calculation in the proof of Theorem~\ref{thmclique} as follows.
Letting $\cF$ denote all potential copies of a single clique $K_k$ in
$G$, take $H',H''\in\cF$ to be two disjoint such copies, arbitrarily
chosen, and define
\[
\Delta_{i,j}:= \mathop{\mathop{\sum_{H\in\cF}}_{\vert V(H)\cap
V(H')\vert
=i}}_{\vert V(H)\cap V(H'')\vert=j}
\P\bigl(H\subset G \mid H',H''\subset
G \bigr),
\]
whence
\[
\Delta_{i,j} = \binomm{n-2k} {k-(i+j)}\binomm{k}i\binomm{k}j
p^{\binom{k}2 - \binom{i}2-\binom{j}2}.
\]
As usual, the probability of encountering a copy of $K_k\cup K_k$ that
does not intersect neither $H'$ nor $H''$ is at most $\P(f_n=1)$, while
the probability of encountering even a single $K_k$ that intersects
$H'$ but not $H''$, conditioned on $H',H''\subset G$, was shown in the
proof of Theorem~\ref{thmclique} to tend to 0. Hence, it remains to
show that $\sum_{2\leq i,j < k}\Delta_{i,j}=o(1)$. The case where
%
%
\begin{equation}
\label{eq-i+j-small} i+j \leq(2-\delta)\frac{\log n}{\log(1/p)}
\end{equation}
for some small $\delta>0$ is treated as in the proof of Theorem~\ref
{thmclique} by writing
\begin{eqnarray*}
\frac{\Delta_{i,j}}{\binom{n}k p^{\binom{k}2}} &\lesssim&\biggl
[\frac
{k^2}{n p^{(\binom{i}2 + \binom{j}2)/(i+j)}} \biggr]^{i+j}
\leq\biggl[\frac{k^2}{n p^{(i+j)/2}} \biggr]^{i+j} \leq\biggl(
\frac{k^{4/\delta}}n \biggr)^{\delta(i+j)/2},
\end{eqnarray*}
which is at most $n^{-2\delta+o(1)}$ by the assumption $i,j\geq2$.
[Note the usage of~\eqref{eq-i+j-small} for the last inequality.]
The complement range for~\eqref{eq-i+j-small} is handled in the
following way. Without loss of generality, assume $i\geq j$, and using
the fact that $\binom{k}2 - \binom{i}2-\binom{j}2 \geq
(k-(i+j))(i+j) +
ij$ we can infer that
\begin{eqnarray*}
\Delta_{i,j} &\leq&\biggl(\frac{e(n-2k)}{(k-(i+j)) \vee1} k p^{i+j}
\biggr)^{k-(i+j)} \bigl( k^{2}p^i \bigr)^j.
\end{eqnarray*}
The first term on the right-hand side is at most $n^{(-1+\delta
+o(1))(k-(i+j))}$ by the assumption on $i+j$, whereas the second term
is at most $n^{(-1+\delta/2+o(1))j}$, which in turn is at most
$n^{-2+\delta+o(1)}$ thanks to the fact that $j\geq2$. Summing these
over $2\leq i,j < k$ now leads to the conclusion that $(f_n)$ is \SNSv.
\end{remark}

\subsection{Proof of Theorem~\texorpdfstring{\protect\ref{thm-strictly-balanced}}{1.6}, part~(1)}

This part of the theorem is a simple consequence of Corollary~\ref
{corsubgraph-WD} via an elegant Poisson approximation argument of
Bollob\'as~\cite{Bollobas}, Theorems~4.1 and~4.3. We include the proof
for completeness.

\begin{lemma}
Let $H_n$ be a strictly balanced graph with $\ell_n \leq\sqrt{\frac
{\log n}{\log\log n}}$ edges, and let $X_n$ count its number of copies
in $G\sim\cG(n,p)$ for $p=p(n)$ such that
\[
0<\liminf_{n\to\infty} \E[X_n]\le\limsup
_{n\to\infty} \E[X_n]<\infty.
\]
Then
%
%
\begin{equation}
\label{eq-strict-balanced-second-moment1} \lim_{n\to\infty} \bigl
(\Var(X_n)-
\E[X_n] \bigr)=0.
\end{equation}
\end{lemma}

\begin{pf}
Denote the number of vertices and edges of $H_n$ by $k$ and $\ell$, and
let $\cF$ denote the set of all potential copies of $H_n$ in $G\sim
\cG(n,p)$.
As before, we break up the second moment of $X_n$ into
\begin{eqnarray*}
\E\bigl[X^2_n\bigr]
&=& \E[X_n]
+\mathop{\sum_{H'\neq H''\in\cF}}_{H'\cap H''= \varnothing} \P
\bigl(H',H''\subset G \bigr) + \mathop{
\sum_{H'\neq H''\in\cF}}_{H\cap H''\neq\varnothing} \P\bigl(H',H''
\subset G \bigr)
\\
&\leq&\E[X_n] + \bigl(1-o(1)\bigr) \bigl(\E[X_n]
\bigr)^2 + \mathop{\sum_{H'\neq H''\in\cF}}_{H\cap H''\neq
\varnothing}
\P\bigl(H', H''\subset G \bigr),
\end{eqnarray*}
where the inequality between the lines used the fact that $k \ll\sqrt
{n}$ as well as the assumption that $\E[X_n]$ is bounded away from 0
and $\infty$, as in the proof of Proposition~\ref{proppoisson-WD}. We
will show below that the summation in the right-hand side is $o(1)$,
which will then imply~\eqref{eq-strict-balanced-second-moment1}.


Given $H'$ and $H''$ whose vertices overlap, put $t = \vert\{v\in
V(H'')\setminus V(H')\}\vert$, whence $0\le t < k$. (Observe that
$t=0$ is possible since
$H'$ and $H''$ can correspond to different copies of $H_n$ even if
their vertex sets are the same.)
The number of vertices in $H'\cap H''$ is therefore $k-t$.

Assume for the moment that $t>0$.
Since $H_n$ is strictly balanced, it follows
that the number of edges of $H''$ between vertices in $V(H')\cap
V(H'')$ is strictly less than
$(k-t)\ell/k$. Thus the number of edges in $H''$ with at least one
endpoint not in
$V(H')\cap V(H'')$ is strictly more than $\ell-(k-t)\ell/k=t\ell/k$.
Since the number of such edges is an integer, there are in fact at
least $t\ell/k+ 1/k$
such edges; hence the number of edges in $H'\cup H''$ is at least
$
\ell+\frac{t\ell+1}{k}
$.
Now, if $t=0$, the number of edges in $H'\cup H''$ is at least $\ell
+1$ (since
$H'\neq H''$). Altogether, this number is always at least
$\ell+(t\ell+1)/k$.

It is easy to see that the third summand is at most
\[
\sum_{s=k}^{2k-1} \binomm{n} {s} \biggl(
\binomm{s} {k}\frac{k!}{a} \biggr)^2 p^{(s\ell+1)/k},
\]
where $a$ denotes the size of the automorphism group of $H_n$, and $s$
corresponds to $k+t$. The last sum is at most
%
%
\begin{equation}
\label{RoshHashana} \sum_{s=k}^{2k-1}
\frac{n^s}{s!} \biggl(\frac{s!}{a} \biggr)^2
p^{(s\ell+1)/k}.
\end{equation}

Note now that
\[
\E[X_n]=\binomm{n} {k}\frac{k!}{a} p^{\ell}=
\bigl(1+o(1)\bigr)\frac{n^kp^\ell}{a}
\]
since $k\ll\sqrt{n}$. It follows that
\[
p=\frac{(a\E[X_n])^{1/\ell}}{n^{k/\ell}}\bigl(1+o(1)\bigr
)^{1/\ell}.
\]
Substituting this back into~\eqref{RoshHashana} yields that the third
sum that we are interested in
is at most
\[
\bigl(1+o(1)\bigr) \sum_{s=k}^{2k-1}
\frac{1}{s!} \biggl(\frac{s!}{a} \biggr)^2 \bigl(a
\E[X_n]\bigr)^{ (s +\ell^{-1} )/k} \frac{1}{n^{1/\ell}}.
\]
Since $a\ge1$ and $s/k + (\ell k)^{-1}\le2$, the above sum is at most
\[
\bigl(1+o(1)\bigr) k \bigl(\E[X_n]^2 \vee1 \bigr)
(2k)! \frac{1}{n^{1/\ell}}.
\]
Since $k \leq\ell+1$, this is at most
\[
\bigl(1+o(1)\bigr) (\ell+1) \bigl(\E[X_n]^2 \vee1
\bigr)\frac{(2\ell+2)!}{n^{1/\ell}}.
\]
It is easy to verify, using the fact that $\E[X_n]$ is bounded away
from 0 and $\infty$ and that
$\ell\le\sqrt{\frac{\log n}{\log\log n}}$, that this last term is
$o(1)$, as desired.
\end{pf}

\subsection{Proof of Theorem~\texorpdfstring{\protect\ref{thm-strictly-balanced}}{1.6}, part~(2)}

Consider $G\sim\cG(n,\lambda/n)$ for some large enough fixed
$\lambda
>1$, and let
$H_n$ be the graph comprised of two triangles connected by a path of length
%
%
\begin{equation}
\label{eq-path-length} r_n = \bigl\lfloor\tfrac{3}2
\log_\lambda n\bigr\rfloor.
\end{equation}
[Any choice of $(1+\delta)\log_\lambda n \leq r_n \leq(2-\delta
)\log
_\lambda n$ would be valid, as will later become evident; we consider
this particular $r_n$ to simplify the presentation.]
It is easy to see that $H_n$ is strictly balanced.
That $\One_{\{H_n\subset G\}}$ is not $\NS$ will follow from the next
two propositions which may be of independent interest.

%
\begin{proposition}
\label{proptriangles-C1-not-NS}
Let $G\sim\cG(n,p)$ for $p=\lambda/n$ with $\lambda\geq4$ fixed, and
let $\cC_1$ be the largest component of $G$.
Define the event
%
%
\begin{equation}
\label{eq-Delta-k-def} \Delta_k = \{ \cC_1 \mbox{ contains at
least $k$ triangles} \}.
\end{equation}
For any fixed $k\geq1$, the function $\One_{\Delta_k}$ is
nondegenerate and not $\NS$.
\end{proposition}

%
\begin{proposition}
\label{proptriangles-C1-paths}
Let $G\sim\cG(n,p)$ for $p=\lambda/n$ where $\lambda>1$ is some large
enough constant, and let $\cC_1$ denote the largest component of $G$.
W.h.p., every pair of triangles in $\cC_1$ is connected by a simple
path of length $r_n=\lfloor\frac{3}2 \log_\lambda n\rfloor$.

Consequently, $\P(H_n\subset G) = \P(\Delta_2)+o(1)$ where $\Delta_2$
is as in~\eqref{eq-Delta-k-def}.
\end{proposition}

Indeed, Proposition~\ref{proptriangles-C1-not-NS} will follow from
showing that the giant component is, in a sense, robust under the noise
operator, hence; for instance, triangles in $\cC_1$ are likely to
remain in the new largest component. The conclusion of Proposition~\ref
{proptriangles-C1-paths} that the properties $\{H_n\subset G\}$ and
$\Delta_2$ are equivalent up to a negligible probability (together with
their nondegeneracy at the given $p=\lambda/n$) will then preclude the
noise sensitivity of $\One_{\{H_n\subset G\}}$.

Our proofs will exploit the well-known fact that the
breadth-first-search exploration process of the component of a given
vertex is well approximated [up to depth $c\log n$ for a suitable
$c(\lambda)$] by a $\Po(\lambda)$-Galton--Watson tree (a supercritical
branching process in our setting), whence belonging to the giant
component would correspond to the survival of this branching process.
Further set $\lambda_\star< 1$ to be the reciprocal of $\lambda$ in that
\[
\lambda e^{-\lambda} = \lambda_\star e^{-\lambda_\star}.
\]
It is known that $\lambda_\star$ equals the probability that,
conditioned on the survival of the branching process, the number of
surviving children of the root is $1$.

\begin{pf*}{Proof of Proposition~\ref{proptriangles-C1-not-NS}}
Let $\{v_1,\ldots,v_n\}$ be the vertices of $G$ arbitrarily ordered,
let $V' = \{v_i\dvtx i \leq\lceil n/10\rceil\}$ and let $G'$ be the
induced subgraph of $G$ on $V'$.
Denoting by $Y$ the number of triangles in $G'$, we note that, as
$G'\sim\cG(n',p')$ with $p'=\lambda/n\sim\lambda/(10n')$ for
$n'=\vert V'\vert$, it is well known [and also follows from the
second moment analysis in the proof\vspace*{1pt} of part~(1) of Theorem~\ref{thm-strictly-balanced}] that
$Y\stackrel{\mathrm{d}}\to\Po(\hat{\lambda})$ for some $\hat
{\lambda
}>0$ fixed (namely, $\hat{\lambda}= \lambda^3 / 6000$).

Next, write $V''= \{ v_i\dvtx i > \lceil n/10 \rceil\}$ and for each
vertex $x\in V'$ let $G''_x$ be the induced subgraph on $V'' \cup\{x\}
$. Further let $\Gamma_t(x)$ denote the exploration process from $x$ in
$G''_x$; that is, for each $t\geq1$
\[
\Gamma_t(x) = \bigl\{y\in V''\dvtx
\dist_{G''_x}(x,y)=t\bigr\}.
\]
This breadth-first-search exploration process up to some time $R$
yields a tree $\cT_x(R)$ which is stochastically dominated by a $\Bin
(0.9n,\lambda/n)$-Galton--Watson tree with $R$ levels (since $\vert
V''\vert\leq0.9n$), and as long as the number of exposed vertices
is $o(n)$ it stochastically dominates a $\Bin(7n/8,\lambda
/n)$-Galton--Watson tree (e.g.) with the same number of levels.

Reveal the graph $G'$, and pick an arbitrary vertex from each triangle
in it, denoting these vertices by $\{x_1,\ldots,x_Y\}$. Set
\[
R:= 10 \log_2 \log n,
\]
and expose $\cT_{x_i}(R)$ for all $i=1,\ldots,Y$ level by level as
described above. An important observation is that, should any of these
trees intersect, it would imply that $G$ contains a subgraph $F_\ell$
consisting of two triangles and a path of length $\ell= O(\log\log
n)$ between them. However, if $\kappa=\kappa(n)$ is any sequence going
to $\infty$ with $n$, then w.h.p. no two triangles in $G$ have distance
less than $ \log_\lambda(n)-\kappa$ between them. Indeed, the
expected number of copies of all graphs $\{ F_{\ell}\dvtx \ell\leq
\log
_\lambda(n)-\kappa\}$, where $F_\ell$ consists of two triangles and a
path of length $\ell$ edges between them, is at most
\[
\sum_{\ell\leq\log_\lambda(n)-\kappa} (np)^6 n^{\ell-1}
p^{\ell} \lesssim\sum_{\ell\leq\log_\lambda(n)-\kappa}
\frac{\lambda^{\ell}}n \lesssim\lambda^{-\kappa} = o(1).
\]
In particular, w.h.p. the $Y$ trees exposed above are pairwise disjoint.
In addition, standard large deviation estimates for the binomial distribution
(cf.~\cite{JLR}, Corollary~2.3) imply that for any given $x$
\[
\P\biggl(\biggl\vert\bigcup_{t\leq R}
\Gamma_t(x)\biggr\vert\geq\lambda^R \biggr) \leq
e^{-c (\log n)^{2}},
\]
where $c>0$ is an absolute constant. [This can be argued, e.g.,
by noting that for small enough $\delta$,
the event $\{\vert\bigcup_{t\leq R}\Gamma_t(x)\vert\geq\lambda
^R\}$ implies that for some $t\leq R$, we must have either
$\{L_t\ge L_{t-1}\mu+ \log^2 n, L_{t-1}\leq\log^2 n\}$
or
$
\{L_t\ge(1+\delta)L_{t-1}\mu, L_{t-1}\ge\log^2 n\}
$, where\vspace*{1pt} $\mu:=7\lambda/8$.]
Therefore, w.h.p. no vertex sees more than $\lambda^R = n^{o(1)}$
vertices by
time $R$, and hence we can define on the same probability space
$
(Y,\cT_{x_1}(R),\ldots,\cT_{x_Y}(R), \cT'_1(R),\ldots,\cT
'_{x_Y}(R) )
$
so that $ (\cT'_1(R),\ldots,\cT'_{x_Y}(R) )$ are
i.i.d. $\Bin(7n/8,\lambda/n)$-Galton--Watson trees with $R$ levels and
such that
$\P(\bigcap_{i=1}^Y \{ \cT'_i(R)\subset\cT_{x_i}(R) \})=1-o(1)$.

Let $\tau_L(d)$ be the probability that a Galton--Watson tree with
offspring distribution $L$ contains a $d$-regular subtree (sharing the
same root). This quantity was expressed in~\cite{PD} as a solution to
an equation involving the p.g.f. of $L$. When $L\sim\Po(\mu)$, it was
shown that $\tau_L(d)$ is the largest solution of $(1-s)\exp(\mu s) =
\sum_{j=0}^{d-1} (\mu s)^j / j!$, which is positive whenever
$d=(1-\varepsilon_\mu)\mu$ for some $\varepsilon_\mu\to0$ as $\mu
\to\infty$;
see Section~4 of that work. For $d=2$, the analysis of~\cite{PD} [and
equations~(4.3),~(4.4) in particular] shows that $\tau_L>0$ provided
$\mu> \exp(y)/y$, where $y$ is the unique positive solution to
$y^2+y+1=\exp(y)$; for example, $\mu> 3.351$ would suffice for a
positive probability of containing a binary subtree.
In case of $L\sim\Bin(n,p)$ (explicitly stated in~\cite{LP}, Section~5),
$\tau_L(d)$ is the largest solution $s\in(0,1]$ of $1-s = \P(\Bin(n,p
s) \leq d-1)$. For $p=\mu/n$, since $L\stackrel{\mathrm{d}}\to\Po
(\mu
s)$ and the intersection of the functions $(1-s)$ and $\exp(-\mu
s)(1+\mu s)$ is not a tangent point for any $\mu$ larger than the
critical one, $\tau_L(d)$ coincides with the Poisson case. Thus in our
setting indeed $\mu=7\lambda/8 \geq3.5$ (by the assumption on
$\lambda
$) suffices for the tree $\cT'_i(R)$ to contain a binary subtree of
height $R$ at its root with positive probability; let $\theta>0$ denote
this probability.

Altogether, it follows that we can define on a common
probability space our random graph and a $\Po(\lambda' \theta)$
variable $Z$
so that w.h.p. the number of triangles in~$G'$,
for which the exploration process into $V''$ from one of the endpoints
contains a
binary subtree of height $R$ rooted at that vertex, is at least $Z$.
Hence, for any fixed $k\geq1$ there will be at least $k$ such
triangles with positive probability
(here we see that $\Delta_k$ is nondegenerate: with positive
probability $G$ is triangle-free, and with positive probability we find
$k$ triangles as above, each one connected to at least $2^{\lfloor
R\rfloor} \asymp(\log n)^{10}$ vertices and thus part of $\cC_1$
w.h.p.; see, e.g., \cite{JLR}, Theorem~5.4).

The proof is completed by noticing that each of these triangles is
robust under the noise operator. Indeed, the triangle itself survives
the noise with probability $(1-\varepsilon)^3$, and henceforth the noise
operator on a binary tree is simply a branching process with offspring
distribution $\Bin(2, 1-\varepsilon)$. Letting $Z_t$ be its population
size at time $t$, a classical fact on supercritical branching processes
whose offspring distribution $L$ has a finite second moment is that, if
$m=\E L>1$ and $q<1$ is the extinction probability, for any fixed
$\delta>0$ with probability $1-q-\delta$, we have that
$\vert Z_R\vert\geq c m^R$ for some fixed $c>0$. Here we have
$m=2(1-\varepsilon)$, yielding that $\vert Z_R\vert\geq c(\log
n)^2$ for a small enough $\varepsilon$, except with probability
$q+\delta
\leq2q$ (for a suitable $\delta$) where $q$ goes to $0$ with
$\varepsilon
$. This would in turn correspond to the scenario where w.h.p. the
triangle under consideration is part of $\cC_1^\varepsilon$, the largest
component of the new graph [as the second largest component has $O_\smP
(\log n)$ vertices]. Altogether, we have shown that for $f_n=\One
_{\Delta_k}$, a positive fraction of the space $\{\omega\dvtx
f_n(\omega
)=1\}$ is such that $\P(f_n(\omega^\varepsilon)=1 \mid\omega) \geq
1-g(\varepsilon)$ where $g(\varepsilon)\to0$ as $\varepsilon\to0$. By
Proposition~\ref{proequivalentNSdefn} it then follows that $(f_n)$ is
not noise sensitive.
\end{pf*}

It remains to prove Proposition~\ref{proptriangles-C1-paths}. While it
is possible to derive the proof from various routine branching process
estimates, it will be convenient to appeal to estimates to this effect
that were developed specifically for the setting of a sparse random
graph $\cG(n,\lambda/n)$ in the recent work of Riordan and
Wormald~\cite
{RW}. Similarly to before, let $ \Gamma_t(x):= \{ v\in V(G)\dvtx
\dist
_G(x,v)=t\}$ for $t\geq0$ be the set of all vertices of $G$ at
distance exactly $t$ from $x$. Set
\[
w:=(\log n)^6,\qquad t_0 = \lfloor
\log_{\lambda_\star^{-1}} n \rfloor,\qquad t_1:= \lfloor
\log_\lambda w \rfloor,
\]
following the notation of~\cite{RW}.
Using these definitions, the following was shown in~\cite{RW}, Lemmas
2.1 and 2.2; see equations~(2.10) and~(2.11) in particular.

\begin{lemma}[(\cite{RW})]\label{lemB1}
Let $0<\kappa=o(\log n)$ be so that $\kappa\to\infty$ with $n$. Then
w.h.p. no vertex $x\in V$ satisfies
$1\leq\vert\Gamma_t(x)\vert< w$ for all $0\leq t \leq
t_0+t_1+\kappa$.
\end{lemma}

Observe that $t_1 = O(\log\log n)$ whereas $t_0 = (1+\delta_\lambda
)\lambda^{-1}\log n$ for $\delta_\lambda$ which approaches $0$ as
$\lambda$ grows. In particular, we have
\[
t_0+t_1+\kappa\leq\tfrac{1}{10}
\log_\lambda n
\]
for large enough $\lambda$ and any sufficiently large $n$.
Therefore, upon defining
\[
\tau_w(x):= \min\bigl\{ t\dvtx \bigl\vert\Gamma_t(x)
\bigr\vert\geq w\bigr\},
\]
we\vspace*{1pt} see that w.h.p. every vertex $x$ satisfies that $x\in
\cC_1$ if and
only if $\tau_w(x) \in[1,\frac{1}{10}\log_\lambda n]$.
We can now address the case $\tau_w(x)\leq\frac{1}{10}\log_\lambda n$,
which will correspond as per the discussion above to every $x$
belonging to the giant component.
%
Here we will need to adapt this conclusion to the case of two
simultaneously growing neighborhoods, as given by the next lemma.

\begin{lemma}\label{lem-B2-pairs}
Fix $\delta>0$ and take $\ell\in\mathbb{N}$ such that $\ell/\log
_\lambda n \in(1+3\delta, 2-2\delta)$. Then w.h.p. every two vertices
$x,y$ whose distance in $G$ exceeds $2\delta\log_\lambda n$ and such
that $\tau_w(x),\tau_w(y)\leq\delta\log_\lambda n$ are connected
by a
simple path of length $\ell$.
\end{lemma}

\begin{pf}
Set $T=\delta\log_\lambda n$, and consider the standard exploration
process which iteratively reveals $\Gamma_t(x)$ for $1\leq t \leq T$.
Estimating $\vert\Gamma_t(x)\vert$ is elementary by standard
concentration arguments, as noted in~\cite{RW}, Lemma~2.4. Indeed,
denoting $L_t = \vert\Gamma_t(x)\vert$ for the number of
vertices at distance $t$ from $x$, clearly
$L_{t+1} \sim\Bin(n-\sum_{i\leq t}L_i, q)$ for $q=1-(1-\lambda
/n)^{L_t} = \lambda L_t / n + O(L_t^2 / n^2)$. It then follows from
large deviation estimates of the binomial variable (as\vspace*{1pt}
used in the
proof of Proposition~\ref{proptriangles-C1-not-NS}) that as long as,
for example, $\sum_{i \leq t} L_i \leq n^{1-\delta/2}$,
\[
\P\biggl(\biggl\vert\frac{L_{t+1}}{\lambda L_t} - 1\biggr\vert \geq\frac
{1}{\log^2 n}
\Big| L_t \biggr) \leq2\exp\biggl(-\biggl(
\frac{1}3-o(1)\biggr)\frac{\lambda L_t}{\log^4 n} \biggr),
\]
where the assumption on $L_t$ makes $\E[L_{t+1} \mid L_t] =
(1+O(n^{-\delta/2}))\lambda L_t$, an approximation error which is
insignificant compared to the $O(1/\log^2 n)$ scale of the deviation
considered here. In particular, we see that necessarily
\[
w \leq L_{\tau_w(x)} \leq2 \lambda w
\]
except with probability $\exp(-c w / \log^4 n)=\exp(-c \log^2 n)$ for
an\vspace*{1pt} absolute constant $c>0$.
Furthermore, by accumulating the $O(1/\log^2 n)$ errors up to time $T
= O(\log n)$, this estimate can be extended throughout this interval
[note that since $T = \delta\log_\lambda n$ this will maintain $L_t
\leq n^\delta$ satisfying the requirement on the size of $\sum_{i\leq
t}\vert\Gamma_i(x)\vert$ with room to spare] to yield
\[
\bigl\vert L_t/ \bigl[\lambda^{t-\tau_w(x)} L_{\tau_w(x)}
\bigr] - 1\bigr\vert\leq\frac{\log\log n}{\log n}\qquad\mbox
{for all }
\tau_w \leq t\leq T
\]
except with probability $\exp(-c \log^2 n)$ for some other absolute
$c>0$ [the factor of $\log\log n$ could have been replaced by any
$\kappa(n)$ going to $\infty$ with $n$].

Now, let us adapt the exploration process to a pair of initial points
$x,y$ as follows.
Denoting the set of neighbors of a set $S$ in $G$ by $N_G(S)$, let
\begin{eqnarray*}
\Gamma'_0 &=& \{x\},\qquad\Gamma'_t=
N_G \bigl(\Gamma'_{t-1} \bigr) \Bigm\backslash
\bigcup_{i<t} \bigl(\Gamma'_i
\cup\Gamma''_i \bigr),
\\
\Gamma''_0 &=& \{y\},\qquad
\Gamma''_t= N_G \bigl(
\Gamma''_{t-1} \bigr) \Bigm\backslash\biggl(
\Gamma'_t \cup\bigcup_{i<t}
\bigl(\Gamma'_i \cup\Gamma''_i
\bigr) \biggr).
\end{eqnarray*}
That is, we expand the neighborhood of $x$ among unvisited vertices
(those that had not yet appeared in any of the neighborhoods) followed
by the same procedure for~$y$, repeatedly.

We clearly have that $\bigcup_{t\leq T}\Gamma'_t$ and $\bigcup
_{t\leq
T}\Gamma''_t$ are disjoint by construction. The hypothesis on the
distance of $x,y$ then implies that $\Gamma'_t=\Gamma_t(x)$ and
$\Gamma
''_t=\Gamma_t(y)$ for all $t\leq T$.
It now follows that
$ \sum_{t\leq T}(\vert\Gamma'_t\vert+\vert\Gamma
''_t\vert) \leq5\lambda w n^{\delta}$
with probability $1-\exp(-c\log^2 n)$ for some absolute $c>0$.

Exposing $\Lambda'_t$ for $t=T+1,\ldots,\lceil\ell/2\rceil$
alternating with exposing $\Lambda''_t$ for $t=T+1,\ldots,\lfloor
\ell
/2\rfloor$,
the exact same concentration argument as above---while recalling that
$\ell<(2-2\delta)\log_\lambda n$ by hypothesis and so at all times
above there are at least $(1-O(n^{-\delta}))n$ unexposed
vertices---implies that with probability $1-\exp(-c\log^2 n)$ for some
absolute $c>0$, we have
\begin{eqnarray*}
\bigl\vert\bigl\vert\Gamma'_t\bigr\vert/
\bigl(\lambda^{t-T} \bigl\vert\Gamma'_{T}
\bigr\vert\bigr) - 1\bigr\vert&\leq&\frac{\log\log n}{\log
n}\qquad\mbox{for all }T
\leq t\leq\lceil\ell/2\rceil,
\\
\bigl\vert\bigl\vert\Gamma''_t
\bigr\vert/ \bigl(\lambda^{t-T} \bigl\vert\Gamma''_{T}
\bigr\vert\bigr) - 1\bigr\vert&\leq&\frac{\log\log n}{\log
n}\qquad\mbox{for all }T
\leq t\leq\lfloor\ell/2\rfloor.
\end{eqnarray*}
Combining this with the fact that $\vert\Gamma'_T\vert,\vert
\Gamma''_T\vert\geq w$ along with the hypothesis $\ell> (1+3\delta
)\log_\lambda n$ now yields that with the aforementioned probability,
\[
\bigl\vert\Gamma'_{\lceil\ell/2\rceil}\bigr\vert\geq
n^{(1+\delta)/2}\quad\mbox{and}\quad\bigl\vert\Gamma''_{\lfloor
\ell
/2\rfloor}
\bigr\vert\geq n^{(1+\delta)/2}.
\]
Finally, observe that none of the potential edges between $\Gamma
'_{\lceil\ell/2\rceil}$ and $\Gamma''_{\lfloor\ell/2\rfloor}$
has been
examined yet, and the probability that none belong to $G$ is
at most
\[
(1-\lambda/n)^{\vert\Gamma'_{\lceil\ell/2\rceil}\vert\vert
\Gamma''_{\lfloor\ell/2\rfloor}\vert} \leq\exp\bigl(-\lambda
n^\delta\bigr).
\]
As any such edge yields a simple path of length $\ell$ between $x,y$,
the proof of the lemma is concluded by a union bound over $x,y$, easily
accommodated by the fact that all error probabilities were
super-polynomially small in $n$.
\end{pf}

With the above ingredients, we can establish Proposition~\ref
{proptriangles-C1-paths} guaranteeing length-specific paths between
triangles in the giant component $\cC_1$.
\begin{pf*}{Proof of Proposition~\ref{proptriangles-C1-paths}}
Since $\cC_1$ is of linear size w.h.p., and thanks to Lemma~\ref
{lemB1} and the discussion following it, w.h.p. every vertex $x\in\cC
_1$ satisfies $\tau_w(x) < \frac{1}{10}\log_\lambda n$. Choosing
$\delta
=\frac{1}{10}$ and $\ell=r_n$ in Lemma~\ref{lem-B2-pairs} we obtain
that w.h.p. every two vertices $x,y\in\cC_1$ with $\dist
_G(x,y)>\frac
{1}{5}\log_\lambda n$ have a simple path connecting them of distance
precisely $r_n=\lfloor\frac{3}2 \log_\lambda n\rfloor$.

The first statement of the proposition now follows from the fact noted
in the proof of Proposition~\ref{proptriangles-C1-not-NS} that for any
$\kappa=\kappa(n)$ going to $\infty$ with $n$, w.h.p. no two triangles
in $G$ have distance less than $ \log_\lambda(n)-\kappa$ between them.
In particular, w.h.p. every pair of triangles in $\cC_1$ has distance
at least $\frac{1}2\log_\lambda n$, and thus are connected by a path of
length $r_n$, as argued above.

Finally, it is well known (see, e.g.,~\cite{JLR}, Theorem 5.12) that
w.h.p. $\cC_1$ is the only component that contains more than a single
cycle, and therefore $\P(H_n \subset G) = \P(H_n\subset\cC_1)+o(1)
\leq\P(\Delta_2)+o(1)$. As we have shown above that $\P(\Delta_2)
\leq
\P(H_n\subset\cC_1) + o(1)$, this completes the proof.
\end{pf*}

Propositions~\ref{proptriangles-C1-not-NS} and~\ref
{proptriangles-C1-paths} combined complete the proof of Theorem~\ref
{thm-strictly-balanced}.\vadjust{\goodbreak}

\section{General properties of strong noise sensitivity}\label{sec0-vs-1}

\subsection{0-strong versus 1-strong noise sensitivity}

The following proposition gives a simple and yet useful necessary
condition for \SNSv.

%
\begin{lemma}\label{lemboundedwitness}
Let $(f_n)$ be a sequence of monotone Boolean functions, and let
$Y_n(\omega) = \sum_{W\in\cW_0(f_n)} \One_{\{\omega_{W}\equiv0\}}$
count\vspace*{1pt} the occurring 0-witnesses in $\omega\in\Omega
_n$. If $\sup_n \E
[Y_n]<\infty$,
then the sequence is not \SNSv.
\end{lemma}

\begin{pf}
Clearly if $W\in\cW_1$ and $W'\in\cW_0$, we must have
$W\cap W'\neq\varnothing$, whence
\[
\P\bigl(\omega^{\varepsilon}_{W'}\equiv0 \mid
\omega_W\equiv1 \bigr)\leq\varepsilon\P\bigl(\omega
^{\varepsilon}_{W'}
\equiv0 \bigr),
\]
and so, by our main assumption, there exists some $C>0$ such that for
all $n$
\[
\sup_{W\in\cW_1} \E\bigl[Y_n\bigl(
\omega^\varepsilon\bigr)\mid\omega_W \equiv1 \bigr]\le C
\varepsilon.
\]
It follows that
\[
\inf_{W\in\cW_1}\P\bigl(f_n\bigl(
\omega^{\varepsilon}\bigr)=1\mid\omega_W\equiv1 \bigr)\ge1-O(
\varepsilon),
\]
and thus the sequence is not \SNSv~(instead, the conditional probability
given any 1-witness is in some sense noise stable, going to 1 as
$\varepsilon\to0$).
\end{pf}

\begin{remark*}
The converse of Lemma~\ref{lemboundedwitness} is false, as the
recursive 3-majority function
demonstrates. We have shown in Section~\ref{subsecstrnoise} that this
function is not \SNSv,
and yet it is easy to see that $\E[ Y_n]$ is not uniformly bounded
(nor is the expected number of 1-witnesses, by symmetry).\vspace*{1pt}
Indeed, if $a_k$ denotes the number of
0-witnesses when there are $n=3^k$ variables, then $a_0=1$ and
$a_{k+1}=3a_k^2$, and so in general $a_k=3^{2^k-1}$. Since a canonical
witness has size $2^k$, we have
$\E Y_n = \frac{1}3(3/2)^{2^k} \to\infty$.
\end{remark*}

Many of the examples that we have seen are \SNSv\ but not \SNSn\ or
vice versa.
We next show that there are Boolean functions which are both.

%
\begin{theorem} \label{thm1SNSand0SNS}
There exists a sequence of monotone nondegenerate\break Boolean functions
which are both
\SNSv\ and \SNSn.
\end{theorem}

\begin{pf}
Define the following Boolean functions:
\begin{itemize}
\item$g_n$: the tribes function on $n$ bits with $\lfloor\log_2
(\frac
{n}{\log_2n})\rfloor$-bit blocks (as usual,
potentially ignoring one shorter block to remedy divisibility issues).
\item$h_n$: the tribes function on $m_n:=\lfloor n^{\log n}\rfloor$ bits
with $b_n:=\lfloor\log_2 (\frac{m_n}{\log_2m_n})\rfloor$ bits\vspace*{1pt} per
block and reversed 0/1 roles
($h_n =0$ if and only if there is an all-0 block).
\item$f_n = g_n\circ h_n$ is the composition of these functions acting
on $m_n n$ bits
(applying $h_n$ to the first $m_n$ bits, the next $m_n$ bits, etc.,
then feeding the $n$ output bits into~$g_n$),
which we claim is both \SNSv\ and \SNSn.
\end{itemize}
Let $p_n$ be such that $\P(h_n=1)=1/2$ [it is
easy to see
that $p_n = 1/2+o(1)$].
The proof will follow from two straightforward properties of $h_n$.

First, we claim that for any $\varepsilon>0$, there exists $\delta>0$
so that
%
%
\begin{equation}
\label{eqBehaviorOfTribes} \inf_n\inf_{W \in\cW_1(h_n)}\P
\bigl(h_n\bigl(\omega^{\varepsilon}\bigr)=0\mid
\omega_W\equiv1 \bigr)\ge\delta.
\end{equation}
Indeed, the number of 0-witnesses occurring in
$\omega^{\varepsilon}$ given $\omega_W\equiv1$ is binomial with parameters
$\Bin((1+o(1))\frac{m_n}{\log_2 m_n}, \varepsilon p^{b_n})$. Since
$p^{b_n} \asymp\frac{\log_2 m_n}{m_n}$,
for fixed $\varepsilon$ this converges to a nontrivial
Poisson distribution, from which~\eqref{eqBehaviorOfTribes} follows.

Second, we argue that for any $\varepsilon> 0$ we have
%
%
\begin{equation}
\label{eqBehaviorOfTribesAGAIN} \max_{W\in\cW_0(h_n)}\P\bigl
(h_n\bigl(
\omega^{\varepsilon}\bigr)=0\mid\omega_W\equiv0 \bigr) -\P
(h_n=0 )=o (1/n ).
\end{equation}
To see this, note that since the 0-witnesses for $h_n$ are disjoint,
the only gain from conditioning on the event
$\omega_W\equiv0$ for some 0-witness $W$ is that the probability that
$\omega^\varepsilon_W\equiv0 $ is increased. Therefore, it suffices to
show that
$
\P(\omega^{\varepsilon}_W\equiv0 \mid\omega_W\equiv0 )=o (1/n )
$ uniformly over $W$.
Indeed this holds as $\P(\omega^{\varepsilon}_W\equiv0 \mid\omega
_W\equiv
0 ) = (1-\varepsilon p_n)^{b_n}$ with $p_n\sim1/2$ and $b_n \gtrsim
\log
m_n \gtrsim\log^2 n$,
thus establishing~\eqref{eqBehaviorOfTribesAGAIN} (with room to spare).

To show that $(f_n)$ is \SNSv, fix $\varepsilon>0$ and note that a
1-witness $W$ for $f_n$ is
obtained by taking a 1-witness $W'$ for $g_n$ and for each $x\in W'$
taking a 1-witness $W''_x$ for $h_n$. By~\eqref{eqBehaviorOfTribes},
$\P(\omega^\varepsilon_x=0\mid\omega_{W''_x}=1)\geq\delta$
for any $x\in W'$ with $\delta(\varepsilon)>0$ fixed.
Thus $\P(\omega_{W'}\equiv1) \leq(1-\delta)^{\vert W'\vert}
\to0$, and since the rest of the blocks of $g_n$ are independent, we get
[following the same argument used to show~\eqref
{eqBehaviorOfTribesAGAIN} above]
that $(f_n)$ is \SNSv.

It remains to show that $(f_n)$ is \SNSn. Fix $\varepsilon>0$, and again
take a 0-witness $W$
for $f_n$ in the form of a 0-witness $W'$ for $g_n$ and accompanying
each $x\in W'$ by a $0$-witness $W''_x$
for $h_n$.
If $\omega_{W}\equiv0$, then~\eqref{eqBehaviorOfTribesAGAIN}
and the fact that $\vert W'\vert\asymp\frac{n}{\log n}$ tell us
that $\omega^\varepsilon_{W'}$
has a distribution whose total variation distance from an i.i.d.
sequence with parameter $1/2$ goes to 0.
With the other blocks of $g_n$ independent, as before this implies that
$(f_n)$ is \SNSn.
\end{pf}

\subsection{Different levels of noise in strong noise sensitivity}

An interesting fact about noise sensitivity, pointed out
in~Section~\ref
{secprelim}, is that
if the criterion~\eqref{eqnNS} for $\NS$ holds for one fixed
$\varepsilon
\in(0,1)$, then it holds for all such $\varepsilon$.
It is then natural to ask whether strong noise sensitivity also
exhibits this behavior. Clearly, if the criterion~\eqref{eqnSNS} for
\SNSv\ holds for one $\varepsilon\in(0,1)$, then it holds
for all $\varepsilon'>\varepsilon$ by monotonicity. However, the next theorem
tells us that in fact~\eqref{eqnSNS} may hold for some $\varepsilon
\in
(0,1)$ and not for some other $\varepsilon'\in(0,\varepsilon)$.

%
\begin{theorem} \label{1SNSDependsOnEpsilon}
There exists a sequence of monotone Boolean functions $(f_n)$
which is \SNSv\ w.r.t. any fixed $\frac{1}4<\varepsilon<1$, while for any
fixed $0<\varepsilon<\frac{1}5$
\[
\lim_{n\to\infty}\inf_{W \in\cW_1(f_n)} \P
\bigl(f_n\bigl(\omega^{\varepsilon}\bigr)=1\mid
\omega_W\equiv1 \bigr)=1.
\]
\end{theorem}

\begin{pf}
Define the following Boolean functions:
\begin{itemize}
\item$r_n$: recursive $5$-majority on $5^{\lfloor1.01 b_n\rfloor}$
variables where $b_n:=\lfloor\log_2 (\frac{n}{\log_2 n})\rfloor$.
\item$g_n$: the tribes function on $n$ bits with $b_n$-bit blocks.
%
\item$f_n = r_n\circ g_n$ is the composition of these two functions,
acting on $n 5^{\lfloor1.01 b_n\rfloor}$ bits,
which we claim will have the desired properties.
\end{itemize}
Choose $p_n$ such $\P(g_n = 1)=1/2$ [recall that this choice has
$p_n=1/2+o(1)$].
In Claim~\ref{clm5Majority} we related the probability that a witness
for $r_n$ survives the noise
to the $k$-iterated function $h(x)$ from that claim, denoted here $h^{(k)}(x)$.
The next claim establishes two simple features of that function.

%
\begin{lemma} \label{BehaviorOfh}
Let $h(x):=-\frac{1}{2}x^3+\frac{3}{4}x^2+\frac{3}{4}x$ as in~\eqref
{eq-hx-def}.
Then we have $h^{(1.01 m)} (\frac{1}{2}+(0.88)^m ) = \frac{1}2 +
o(1)$ whereas
$h^{(1.01 m)} (\frac{1}{2}+(0.89)^m )=1-o(1)$.
\end{lemma}

\begin{pf}
Letting\vspace*{1pt} $L$ be the linear function
$L(x):=\frac{9}{8}(x-\frac{1}{2})+\frac{1}{2}$, we have $h \leq L$ on
$[\frac{1}2,1]$ since $h$ is concave in that interval and has $h(\frac
{1}2)=\frac{1}2$ and $h'(\frac{1}{2})=\frac{9}{8}$. Since $h$ is
increasing and sends
$[\frac{1}{2},1]$ to itself, it\vspace*{1pt} follows that $h^{(k)}\le
L^{(k)}$ on
$[\frac{1}{2},1]$ for all~$k$.
Observing that $L^{(k)}(x)=(\frac{9}{8})^k(x-\frac{1}{2})+\frac{1}{2}$,
in particular we have
$h^{(1.01 m)} (\frac{1}{2}+(0.88)^m )-\frac{1}2 \leq(\frac
{9}{8})^{1.01 m}(0.88)^{m}\to0$ as $m\to\infty$.

For the second statement, choose $p_0\in(\frac{1}{2},1)$ so that
$h'(p_0)=\frac{9}{8}-\frac{1}{1000}$. Since $h$ is concave on $[\frac
{1}{2},1]$,
now $h\ge M$ on $[\frac{1}{2},p_0]$ where $M$ is the linear function
$M(x):=h'(p_0)(x-\frac{1}{2})+\frac{1}{2}$.
Since $h$ is increasing and sends $[\frac{1}{2},1]$ to itself,
$h^{(k)}(x)\ge M^{(k)}(x)$ for all $x$ and $k$ satisfying
$M^{(k-1)}(x)\le p_0$ (i.e., until the orbit of $x$ passes $p_0$).
Since $M^{(m)}(x)=(h'(p_0))^m(x-\frac{1}{2})+\frac{1}{2}$, we have
$M^{(m)}(\frac{1}{2}+(0.89)^m) \to\infty$, and so
$h^{(m)}(\frac{1}{2}+(0.89)^m)\ge p_0$ for\vspace*{1pt} large $m$.
Since $p_0$ is a
fixed number larger than $1/2$, and $h(x)$ has fixed points at $\{
0,1/2,1\}$, the additional $m/100$ iterations give $h^{(1.01
m)}(x)=1-o(1)$, as required.
\end{pf}
As for the tribes function $g_n$, it is easy to check that for any
1-witness $W$,
\[
\Gamma_n:= \P\bigl(g_n\bigl(\omega^{\varepsilon}
\bigr)=1\mid\omega_W\equiv1 \bigr)-\P(g_n=1)=
u_n \bigl[ \bigl(1-\varepsilon(1-p_n)
\bigr)^{b_n}-p_n^{b_n} \bigr],
\]
where $u_n$ is the probability that none of the blocks except possibly
the first one is an all 1-block,
which is $1/2+o(1)$. As $p_n= 1/2+o(1)$, it follows, say, that for any
fixed $0<\varepsilon<1$, any sufficiently large $n$ and any
$1$-witness $W$,
%
%
\begin{equation}
\bigl(1-\varepsilon/2-\varepsilon^2/16 \bigr)^{b_n} \leq
\Gamma_n \leq\bigl(1-\varepsilon/2+\varepsilon^2/16
\bigr)^{b_n}.\label{eq-1SNSQuantitativeTribes}
\end{equation}
Any 1-witness $W$ for $f_n$ is obtained by taking some 1-witness $W'$
for $r_n$
together with a 1-witness $W''_x$ for $g_n$ for every $x\in W'$.
By~\eqref{eq-1SNSQuantitativeTribes}, for large enough $n$ the
distribution of the bits $\omega^\varepsilon_{W'}$ is i.i.d. with
probability $q_n$ of $1$, where $q_n \leq1/2 + (0.88)^{b_n}$ if
$\varepsilon>\frac{1}4$, whereas $q_n \geq\frac{1}2 + (0.89)^{b_n}$ if
$\varepsilon<\frac{1}5$.

Finally, the analysis in Claim~\ref{clm5Majority} tells us that for
recursive 5-majority with $k$ levels on an input distribution that is
i.i.d. $(q,1-q)$ for $q\neq1/2$ on a 1-witness $W'$ and i.i.d.
$(1/2,1/2)$ elsewhere, the
probability that the output is 1 is $h^{(k)}(q)$. This fact together with
Lemma~\ref{BehaviorOfh} completes the proof.
\end{pf}

\section*{Acknowledgments}
This work was carried out when Jeffrey E. Steif
was visiting Microsoft Research at Redmond, and he thanks
the Theory Group for its hospitality and for creating a stimulating
research environment.
We thank the anonymous referees for useful comments.




%

\printaddresses
\end{document}